\documentclass[a4paper,11pt]{amsart}

\usepackage{amsmath}
\usepackage{amssymb}
\usepackage{amsthm} 
\usepackage{amsfonts,amsbsy}
\usepackage{hyperref}
\usepackage{enumerate}

\usepackage[margin=3.5cm]{geometry}

\theoremstyle{plain}
\newtheorem{Thm}{Theorem}[section]
\newtheorem{Lemm}[Thm]{Lemma}
\newtheorem{Cor}[Thm]{Corollary}
\newtheorem{Prop}[Thm]{Proposition}
\newtheorem{Fact}[Thm]{Fact}
\newtheorem{Conj}[Thm]{Conjecture}

\theoremstyle{definition}
\newtheorem{Def}[Thm]{Definition}
\newtheorem{Rem}[Thm]{Remark}

\newcommand{\Z}{\mathbb{Z}} 
\newcommand{\N}{\mathbb{N}}
\newcommand{\Q}{\mathbb{Q}} 
\newcommand{\R}{\mathbb{R}}

\newcommand{\alg}{^{\text{alg}}}

\newcommand{\X}{\mathcal{X}}
\newcommand{\Y}{\mathcal{Y}}
\newcommand{\ZZ}{\mathcal{Z}}
\newcommand{\B}{\mathcal{B}}

\newcommand{\F}{\mathcal{F}}

\newcommand{\CC}{\mathcal{C}}

\newcommand{\PP}{\mathbb{P}}

\newcommand{\Alg}{\mathbb{E}}
\newcommand{\A}{\mathcal{A}}
\newcommand{\e}{e}

\newcommand{\spanEndA}{\operatorname{span}_{\End(\Alg)}}
\newcommand{\spank}{\operatorname{span}}
\newcommand{\locus}{\operatorname{locus}}

\newcommand{\Fin}{\operatorname{Fin}}
\newcommand{\Sub}{\operatorname{Sub}}

\newcommand{\Mat}{\operatorname{Mat}}
\newcommand{\Tor}{\operatorname{Tor}}
\newcommand{\str}{\lhd}

\newcommand{\acl}{\operatorname{acl}}

\newcommand{\ld}{\operatorname{lin.d.}}

\newcommand{\trd}{\operatorname{tr.d.}}

\newcommand{\RM}{\operatorname{RM}}
\newcommand{\RU}{\operatorname{RU}}
 
\newcommand{\tp}{\operatorname{tp}}
\newcommand{\qftp}{\operatorname{qf-tp}}
\newcommand{\cl}{\operatorname{cl}}
\newcommand{\scl}{\operatorname{scl}}
\newcommand{\End}{\operatorname{End}}

\newcommand{\Aut}{\operatorname{Aut}}

\newcommand{\id}{\operatorname{id}}

\newcommand{\dd}{\operatorname{d}}

\newcommand{\fin}{_\text{fin}} 

\newcommand{\eps}{\epsilon} 
\newcommand{\subsetfin}{\subset_{\fin}}

\newlength {\xxxIndep}
\newlength {\yyyIndep}

\newcommand{\IndepOperator}{\operatorname{\hspace*{-0.4em}\mbox{\raisebox{-0.5ex
}[0pt][0pt]{$\overset{\textstyle\mid}{\smash{\smile}}$}}\hspace*{-0.4em}}}

\newcommand{\INDEP}[2][{\
}]{\hspace*{0.4em}\sideset{}{_{}^{#1}}\IndepOperator\limits_{#2}}

\begin{document}

\title{Theories of green points}
\author{Juan Diego Caycedo}
\address{Mathematisches Institut, Albert-Ludwigs-Universit\"at Freiburg, Eckerstr. 1, 79104 Freiburg, Deutschland.}
\email{juan-diego.caycedo@math.uni-freiburg.de}
\date{December 17, 2013}

\begin{abstract}
Poizat's construction of theories of fields with a multiplicative subgroup of \emph{green points} is extended in several directions: First, we also construct similar theories where the green points form a divisible $\End(E)$-submodule of an elliptic curve $E$. Second, the subgroup (submodule) of green points is allowed to have torsion, whereas in Poizat's work this more general case was only dealt with assuming the unproven Conjecture on Intersections with Tori (CIT). Third, motivated by Zilber's work on connections with non-commutative geometry, we construct a version of the theories of green points in the multiplicative group case for which the distinguished subgroup is not divisible but a $\Z$-group, which we call theories of \emph{emerald points}. In a subsequent joint paper with Boris Zilber we find natural models for the constructed theories.
\end{abstract}

\maketitle


\bibliographystyle{alpha}

\section{Introduction}

In this paper, Hrushovski's predimension method is used to construct certain theories of expansions of the natural algebraic structure on the multiplicative group and on an elliptic curve by a predicate for a subgroup. 
In the case of the multiplicative group, these are the theories of fields with green points constructed in \cite{PzEq3} by Poizat. 
Following the convention introduced by Poizat, we call the elements of the distinguished subgroup \emph{green points} and the other elements of the structure \emph{white points}, also in the elliptic curve case. The constructed theories will be called \emph{theories of green points}.
The essential property of these theories is that the dimension of the distinguished subgroup is half of that of the ambient algebraic group. Of course, this dimension is not of algebro-geometric nature but rather is the model-theoretic notion of Morley rank, and the distinguished subgroup is indeed far from being an algebraic subgroup.


In fact, Poizat constructed several $\omega$-stable theories of expansions of algebraically closed fields by a unary predicate such that the Morley rank of the field is twice the Morley rank of the predicate. In the initial construction (\cite{PzEq2}), the distinguished subset is not required to have any special properties from the outset. In this case, he called the elements of this subset \emph{black points} and the other elements of the domain \emph{white points}, in the resulting theory the Morley rank of the predicate is $\omega$ and the Morley rank of the domain is $\omega \cdot 2$, and he also carried out a further construction, called \emph{the collapse}, obtaining a theory of Morley rank two, where the predicate has rank one. Later Poizat also constructed similar theories in which the interpretation of the predicate is a subgroup of the additive group of the field, case in which he called the elements of the subgroup \emph{red points}, and theories in which the interpretation of the predicate is a multiplicative subgroup, case in which he called the elements of the subgroup \emph{green points} (\cite{PzEq3}). In the case of red points, his theory has Morley rank $\omega \cdot 2$ in the case of fields of positive characteristic and rank $\omega^2 \cdot 2$ in the characteristic zero case. In the green case, the field always has characteristic zero and the theory has Morley rank $\omega \cdot 2$.


Let us discuss the main reasons for interest in the theories of fields with green points.
The first reason is that the construction has direct relevance for the question of existence of so-called \emph{bad fields}. A bad field is an expansion of an algebraically closed field by a predicate for a multiplicative subgroup whose Morley rank is finite and greater than one. The question of whether such structures exist arose in work on the Algebraicity Conjecture due to Cherlin and Zilber, which states that every simple group of finite Morley rank is an algebraic group over an algebraically closed field and which is open since the late 1970s.
The non-existence of bad fields would have simplified some of the work on the conjecture (see, for example, the introduction of \cite{BadField}). 
Eventually, however, it was suspected that bad fields of characteristic zero could be obtained by Hrushovski's method. Poizat's construction of fields with green points was effectively the first step in this direction. It produced an analogue of a bad field of infinite rank. Moreover, carrying out the corresponding collapse would complete the construction of a bad field. This was subsequently done by Baudisch, Hils, Mart\'in Pizarro and Wagner in \cite{BadField}.
Let us also mention that, by a result of Wagner (\cite{WagnerBad}), the existence of bad fields of positive characteristic $p$ would imply that there are only finitely many $p$-Mersenne primes. Bad fields of positive characteristic are therefore not expected to exist, but proving their nonexistence is, for the same reason, difficult.

The second reason is that the construction of the theories of fields with green points involves strong algebro-geometric results in an essential way. Besides being interesting in its own right, this technical difficulty also exposes connections with Zilber's construction of fields with pseudo-exponentiation, where similar results have to be applied.
In both cases questions of intersections of algebraic subvarieties of algebraic tori with algebraic subgroups naturally appear. In this regard, Zilber's Conjecture on intersections with tori (CIT), from \cite{ZCIT}, answers all the questions involved, but its general validity remains an open question. The conjecture was indeed motivated by Zilber's work on the model theory complex exponentiation.
A partial result, usually called Weak CIT, follows from a theorem of Ax in differential algebra by a model-theoretic argument, as showed by Poizat in \cite{PzEq3} and by Zilber in \cite{ZCIT}. Remarkably, for the construction of theories of green points, Weak CIT suffices. This was shown by Poizat in \cite{PzEq3} in the case where the theory requires the green subgroup to be torsion-free and he stated the question of whether this was also true for arbitrary torsion in the green subgroup. We answer this question positively in Section 4.
Another strong algebraic result that is needed in the construction is the Thumbtack Lemma, which was first proved by Zilber in \cite{ZCovers} (see also \cite{BZCovers}) and which is now available in a very general form by a theorem of Bays, Gavrilovich and Hils \cite{BGH}. The Thumbtack Lemma is used in an essential way to prove that the theories of green points are $\omega$-stable. This point is not explicit in Poizat's original paper and has only become clear in more recent works on the topic, e.g. \cite{HilsGreenGen}.


A third reason is provided by \cite{ZBicol}, where Zilber shows that, assuming Schanuel's conjecture, a certain explicitly given expansion of the complex field by a subgroup of the multiplicative group is a model of Poizat's theory of fields with green points. In an upcoming joint paper with Zilber, we correct, improve and extend the arguments from \cite{ZBicol} in order to find natural models for the theories of green points constructed in this paper. In general we will still have to assume a form of Schanuel's conjecture, or its analogue for the exponential function of the elliptic curve in question, but in generic (in a certain sense) cases our results will be unconditional.


The paper is organised as follows: In Section~\ref{sec:str}, the theories of green points are obtained as the complete theories of certain structures that are constructed there. 
An axiomatization of the theories is then found in Section~\ref{sec:theory} and their main model-theoretic properties are established in Section~\ref{sec:mth}. 
In Section~\ref{sec:eme}, we construct a version of the theories of green points in the multiplicative group case for which the distinguished subgroup is not divisible but a $\Z$-group, which we call theories of \emph{emerald points}. This part is motivated by work of Zilber relating certain expansions of the complex numbers, which will be models of these theories, to the non-commutative tori from non-commutative geometry. 


\textbf{Acknowledgements:} The results in this paper were part of my doctoral thesis at the University of Oxford under the supervision of Boris Zilber. This research was funded by the Marie Curie Research Training Network MODNET. I would like to thank Martin Bays, Martin Hils, Philipp Hieronymi and Boris Zilber for very helpful discussions.

\section{Structures} \label{sec:str}

In this section, we start by introducing the classes of structures that will be considered throughout the paper and, in this context, the usual tools of predimension constructions. We then show that, after fixing some initial data,  each of these classes contains a certain kind of ``generic'' structures, called \emph{rich structures}, all of which satisfy the same complete theory. The complete theories obtained in this way will be the \emph{theories of green points}, which will then be studied in Sections \ref{sec:theory} and \ref{sec:mth}.

\subsection{Preliminaries}

\subsubsection{Algebraic groups}

Let $\Alg$ be the multiplicative group or an elliptic curve over a field $k_0$ of characteristic 0. We use additive notation for the group operation on $\Alg$. 
We denote by $\End(\Alg)$ the ring of algebraic endomorphisms of $\Alg$.
%



Let $A = \Alg(K)$ with $K$ an algebraically closed field extending $k_0$. 
Note that $A$ is an $\End(\Alg)$-module. We have a dimension function on $A$ given by the $\End(\Alg)$-linear dimension, which we denote by $\ld_{\End(\Alg)}$, or simply by $\ld$. We use $\langle Y \rangle$ or $\spanEndA(Y)$ to denote the $\End(\Alg)$-span of a subset $Y$ of $A$. 

Since $K$ is algebraically closed, $A$ is divisible. Also, the ring $\End(\Alg)$ is an integral domain and $k_{\Alg} := \End(\Alg) \otimes_{\Z} \Q$ is its fraction field. The quotient $A/\Tor(A)$ is a $k_{\Alg}$-vector space and for every $Y \subset A$, $\ld_{\End(\Alg)}(Y)$ equals the $k_{\Alg}$-linear dimension of $\phi(Z)$ in $A/\Tor(A)$, where $\phi:A \to A/\Tor(A)$ is the quotient map. The pregeometry on $A/\Tor(A)$ given by the $k_{\Alg}$-span induces a pregeometry on $A$ that we shall denote by $\spank$; this means that for $Y \subset A$, $\spank(Y) = \phi^{-1}(\operatorname{span}_{k_{\Alg}} (\phi(Y)))$.

The algebraic group structure on $A$ induces an algebraic group structure on each cartesian power $A^n$, $n \geq 1$. There is the following characterisation of the algebraic subgroups of $A^n$; for the case of the multiplicative group a proof can be found, for example, in \cite[Section 3.2]{Bombieri}, for the elliptic curve case we refer to \cite[Lemma 1]{Viada}.

Every algebraic subgroup $C$ of $A^n$ is defined by a system of equations of the form 
\[
m \cdot y = 0,
\]
where $y$ is a $n$-tuple of variables and $m$, and $m$ is a $k\times n$-matrix with entries in $\End(\Alg)$, for some $k$. 

Furthermore, if $m$ has rank $k$, then $C$ has dimension $n-k$ (as a Zariski closed set). In this case we say that $C$ has \emph{codimension} $k$.

It follows that a coset $\alpha + C$ of an algebraic subgroup $C$ is defined by a system of equations of the form
\[
m \cdot y = b,
\]
with $y,m$ as before and $b \in A^k$. 

Note that for any algebraically closed intermediate field $k_0 \subset K' \subset K$ we have: the coset $\alpha + C$ is defined over $K'$, if and only if, $\alpha + C$ has a $K'$-rational point, if and only if, $\alpha + C = \alpha' + C$ for some $\alpha'\in \Alg(K')^n$.

\subsubsection{Algebraic varieties as structures}

Let $\Alg$ be affine or projective algebraic variety defined over a field $k_0$.

Let $L_{\Alg}$ be the first-order language consisting of an $n$-ary predicate for each subvariety of $\Alg^n$ defined over $k_0$, $n \geq 1$.

Let $K$ be an algebraically closed field extending $k_0$. Consider the natural $L_{\Alg}$-structure on $\Alg(K)$:
\[
(\Alg(K), (W(K))_{W \in L_{\Alg}}).
\]

Let $T_{\Alg}$ denote the complete theory of this structure. 

The following is the key fact about $T_{\Alg}$ that we shall use. It is well-known, but finding a detailed proof in the literature seems not so easy, there is one such in \cite[Fact A.2.1]{Bays}.

\begin{Fact} \label{fact-biin}
Let $ACF_{k_0}$ be the theory of algebraically closed fields extending $k_0$ in the expansion of the field language by constants for each element of $k_0$. Then, $T_{\Alg}$ is bi-interpretable with $ACF_{k_0}$.
\end{Fact}

We consider $\Alg$ as embedded in affine or projective space, so we have a natural interpretation of $T_{\Alg}$ in $ACF_{k_0}$. In fact, there is an interpretation of $ACF_{k_0}$ in $T_{\Alg}$ that gives the bi-interpretability when paired with the natural interpretation of $T_{\Alg}$ in $ACF_{k_0}$. 

Fact~\ref{fact-biin} has several relevant consequences: the theory $T_{\Alg}$ does not depend on the choice of $K$, and every model $A$ of $T_{\Alg}$ is of the form $A = \Alg(K)$ for a corresponding algebraically closed field $K$ extending $k_0$. Also, $T_{\Alg}$ has quantifier elimination. 

Suppose $\Alg$ is irreducible and 1-dimensional. Then $T_{\Alg}$ is strongly minimal.
Therefore $\acl_{T_{\Alg}}$ induces a pregeometry on every model $A$ of $T_{\Alg}$. By the bi-interpretability, $\acl_{T_{\Alg}}^{\text{eq}}$ equals $\acl_{ACF_{k_0}}^{\text{eq}}$.  From this we get that for any tuple of elements $b$ of a model of $T_{\Alg}$, the $\acl_{T_{\Alg}}$-dimension of $b$ equals the $\acl_{ACF_{k_0}}^{\text{eq}}$-dimension of $b$ when viewed as living in $K^{\text{eq}}$, where $K$ is the field corresponding to $A$. The $\acl_{ACF_{k_0}}^{\text{eq}}$-dimension of $b$ equals the transcendence degree over $k_0$ of any normalised representation of $b$ in homogeneous coordinates, which we shall denote by $\trd(b/k_0)$, or simply by $\trd(b)$, since $\Alg$ and $k_0$ will be fixed throughout.

\subsection{Structures}

Let $L = L_{\Alg} \cup \{G\}$ be the expansion of the language $L_{\Alg}$ by a unary predicate $G$. 

Let $\CC$ be the class of all $L$-structures $\A = (A,G)$ where $A \vDash T_{\Alg}$ and $G$ is a divisible $\End(\Alg)$-submodule of $A$. Note that the class $\CC$ is elementary.  

Following a convention introduced by Poizat, given an $L$-structure $\A = (A,G)$ in $\CC$, we call the elements of $G$ \emph{green points} and the elements of $A \setminus G$ \emph{white points}.

Let $\A = (A,G)$ be a structure in $\CC$. Let $\cl_0^{\A}$ be the pregeometry on $A$ induced by the pregeometry  given by the $k_{\Alg}$-linear span on the quotient $A/\Tor(A)$.

Note that if $\A,\B$ are structures in $\CC$ with $\A \subset \B$, then $A$ is $\cl_0^{\B}$-closed and for all $Y \subset A$, $\cl_0^{\A}(Y) = \cl_0^{\B}(Y)$. Therefore we shall write simply $\cl_0$ for $\cl_0^{\A}$ for any $\A$.  

Notice that the dimension function associated to $\cl_0$ coincides with the $\End(\Alg)$-linear dimension. Henceforth we shall refer to the $\cl_0$-dimension simply as linear dimension and write $\ld$ for it. Moreover, 
for a subset $X$ of $A$, we also denote the set $\cl_0(X)$ by $\spank(X)$. Also notice that for any subset $X$ of $A$, we have $\cl_0(A) \subset \acl_{T_{\Alg}}(X)$, and therefore $\trd(X) \subset \ld(X)$.

Let us denote by $\Sub \CC$ the class of substructures of structures in $\CC$ whose domain is a $\cl_0$-closed set.


Let us now recall some notions related to pregeometries, or more generally, closure operators (i.e. operators on a power set satisfying the pregeometry axioms except for possibly the exchange axiom).

For any closure operator $\cl$ on a set $A$ and any subset Y of $A$, the \emph{localisation of $\cl$ at $Y$} is the closure operator $\cl_Y$ given by $\cl_Y(X) = \cl(X \cup Y)$. We also write $\cl(X/Y)$ for $\cl_Y(X)$. If $\cl$ is a pregeometry with associated dimension function $\dd$, then $\cl_Y$ is also a pregeometry and we denote its associated dimension function by $\dd_Y$. We also write $\dd(X/Y)$ for $\dd_Y(X)$, and refer to this value as the \emph{dimension of $X$ over $Y$}. 

If $Y$ is of finite $\cl$-dimension, then we have the formula $\dd_Y(X) = \dd(XY) - \dd(Y)$ for any $X$ of finite $\cl$-dimension. Also, for an arbitrary subset $Y$ of $A$, we have: for any $X \subsetfin A$,
\[
\dd_Y(X) := \min_{Y' \subsetfin Y} \dd(X/Y').
\]

The dimension function $\dd$ of a pregeometry $\cl$ on $A$ the following \emph{addition formula} holds: for all $X,Y,Z \subsetfin A$ with $X \supset Y \supset Z$,
\[
\dd(X/Z) = \dd(X/Y) + \dd(Y/Z).
\]

A pregeometry $\cl$ is said to be \emph{modular} if the lattice of $\cl$-closed sets is a modular. Equivalently, $\cl$ is modular if the corresponding dimension function $\dd$ has the property that for all $\cl$-closed sets $X,Y$,
\[
\dd(X \lor Y/ Y) = \dd(X / X \land Y),
\]
where the operations $\lor$ and $\land$ are taken in the lattice of $\cl$-closed sets. The pregeometry given by the linear span in a vector space is always modular. From this last fact, it is easy to see that $\cl_0$ is a modular pregeometry. Notice that in the lattice of $\cl_0$-closed sets $X \lor Y = X + Y$ and $X \land Y = X \cap Y$.

\subsection{The predimension function}


Consider the predimension function
\footnote{The expression \emph{predimension function} refers to functions that have a special role in our constructions, it does not, however, imply any properties of the function in question.}
$\delta^{\A}$ defined on the finite dimensional $\cl_0$-closed subsets $Y$ of a structure $\A \in\CC$ by
\[
\delta^{\A}(Y) = 2 \trd(y) - \ld(Y \cap G).
\]
Notice that if $\B$ is also a structure in $\CC$ with $\A \subset \B$, then $\delta^{\A}(Y) = \delta^{\B}(Y)$. Thus, we shall henceforth drop the superindex and write simply $\delta$. For any subset $Y$ of $A$ of finite $\cl_0$-dimension (not necessarily closed), we define the value of $\delta(Y)$ to be $\delta(\cl_0(Y))$.


For any finite dimensional $\cl_0$-closed set $Y$, the \emph{localisation of $\delta$ at $Y$}, $\delta_Y$, is the function given by 
\[
\delta_Y(X) = \delta(XY) - \delta(Y),
\]
for any finite dimensional $\cl_0$-closed set $X$. We also write $\delta(X/Y)$ for the value $\delta_Y(X)$, and call it the \emph{predimension of $X$ over $Y$}. For any finite dimensional $Y$, not necessarily $\cl_0$-closed, we define $\delta_Y$ to be $\delta_{\cl_0(Y)}$.


The following two lemmas establish the main properties of the predimension function. The first one is an immediate consequence of the addition formulas for the dimension functions $\trd$ and $\ld$ we have the following.


\begin{Lemm}
The predimension function $\delta$ satisfies the following \emph{addition formula}: for all finite dimensional $\cl_0$-closed sets $X,Y,Z$ with $X \supset Y \supset Z$,
\[
\delta(X/Z) = \delta(X/Y) + \delta(Y/Z).
\]
\end{Lemm}

\begin{Rem}
It is easy to see that the above lemma can be rephrased as follows: for all finite dimensional $\cl_0$-closed sets $X,Y,Z$, 
\[
\delta(X \lor Y/Z) = \delta(Y/Z) +\delta(X/Y \lor Z).
\]
We shall use the expression \emph{addition formula} to refer to either of the two equivalent formulations.
\end{Rem}



\begin{Lemm}
The predimension function $\delta$ is \emph{submodular} with respect to the pregeometry $\cl_0$, i.e. for all finitely generated $\cl_0$-closed sets $X,Y$, 
\[
\delta(X \lor Y/ Y) \leq \delta(X / X \land Y),
\]
with the operations $\lor$ and $\land$ taken in lattice of finitely generated $\cl_0$-closed subsets of $A$.
\end{Lemm}
\begin{proof}
Note that $\trd$ is submodular with respect to $\cl_0$ (indeed, the dimension function of a pregeometry $\cl$ is always submodular with respect $\cl$ and also with respect to any weaker pregeometry, i.e. any pregeometry $\cl'$ such that $\cl'(X) \subset \cl(X)$ for all $X$). 
Hence for all finite dimensional $\cl_0$-closed sets $X,Y$,
\begin{equation} \label{eq:submodG1}
\trd(X + Y / Y) \leq \trd(X/ X \cap Y).
\end{equation}
Let us look at the negative term in the definition of $\delta_G$ and consider the function $\ld_G(y):= \ld(\spank(y)\cap G)$, for $y \subsetfin A$. We note that it satisfies the following \emph{supermodularity} property: for all finite dimensional $\cl_0$-closed sets $X,Y$,
\begin{equation} \label{eq:submodG2}
\ld_G(X + Y/Y) \geq \ld_G(X/X \cap Y).
\end{equation}
One can see this as follows:
\begin{align*}
\ld_G( X/X\cap Y ) &= \ld( X\cap G/ X\cap Y\cap G )\\
                   &= \ld( X\cap G/ (X\cap G) \cap (Y\cap G) )\\
                   &= \ld( (X\cap G) + (Y\cap G) / Y\cap G )
                   \ \text{ (by modularity of $\spank$)}\\
                   &\leq \ld( (X+Y)\cap G / Y\cap G ).
\end{align*}
Thus, for all finite dimensional $\cl_0$-closed sets $X,Y$, we have, substracting \ref{eq:submodG2} from \ref{eq:submodG1},
\begin{equation} \label{eq:submodG}
\delta(X + Y / Y) \leq \delta(X/ X \cap Y).
\end{equation}
\end{proof}

\begin{Rem}
Let us note that the submodularity of the predimension function $\delta$ with respect to $\cl_0$ can also be expressed as follows: for all finite $x,y \subset A$,  
\[
\delta(xy/ Y) \leq \delta(x / X \cap Y),
\]
where $X = \cl_0(x)$ and $Y=\cl_0(Y)$.
\end{Rem}





\subsection{Strong sets}


\begin{Def}
Let $\A \in \CC$ and let $X$ be a $\cl_0$-closed subset of $A$. 

A finite dimensional $\cl_0$-closed subset $Y$ of $X$ is \emph{strong} in $X$, written $Y \str X$, if for all $x \subsetfin X$, $\delta(x/Y) \geq 0$.

An arbitrary $\cl_0$-closed subset $Y$ of $X$ is strong in $X$, also written $Y \str X$, if it is the union of a directed system of finite dimensional strong $\cl_0$-closed subsets of $X$ with respect to inclusions. Equivalently, $Y$ is strong in $X$ if the collection of all its finite dimensional $\cl_0$-closed subsets that are strong in $X$ is directed with respect to inclusions and has $Y$ as union.

For arbitrary subsets $Y$ of $X$, we say that $Y$ is strong in $X$ if $\cl_0(Y)$ is strong in $X$ in the above sense.
\end{Def}

The following lemma gathers the main properties of strong sets. The proofs of these facts are standard arguments using the submodularity of $\delta$ with respect to $\cl_0$. 

\begin{Lemm} \label{ss-props}
\begin{enumerate}
\item \label{ss-trans} (Transitivity) Let $X,Y,Z$ be $\cl_0$-closed subsets of $A$. If $Z \str Y$ and $Y \str X$ then $Z \str X$.
\item \label{ss-unions} (Unions of chains) Let $(X_i)_{i\in I}$, be an increasing $\str$-chain of subsets of $A$, i.e. for all $i,j \in I$, if $i \leq j$ then $X_i \str X_j$. Put $X = \bigcup_{i \in I} X_i$. Then for all $i \in I$, $X_i \str X$.
\item \label{ss-inter} (Intersections) If $Y_1$ and $Y_2$ are finite dimensional strong $\cl_0$-closed subsets of $A$, then so is $Y_1 \cap Y_2$.
\end{enumerate}
\end{Lemm}

\begin{Def}
Let $X \subset A$. The \emph{strong closure of $X$}, denoted $\scl(X)$, is the smallest $\cl_0$-closed strong subset of $A$ containing $X$.
\end{Def}

\begin{Lemm}
Let $\delta$ be a submodular predimension function on $A$ with respect to the modular pregeometry $\cl_0$. Assume the values of $\delta$ are bounded from below in $\Z$.
Then for every set $X$, the strong closure of $X$ exists.
\end{Lemm}
\begin{proof}
Assume $X$ has finite $\cl_0$-dimension. Since the values of $\delta$ are bounded from below in $\Z$, among all the finite dimensional $\cl_0$-closed sets $Y$ containing $X$ we can find one such that $\delta(Y)$ is minimal, and such $Y$ is then clearly strong. Thus, the collection $\mathcal{S}_X$ of all strong finite dimensional $\cl_0$-closed sets containing $X$ is non-empty. By Lemma~\ref{ss-props}.\ref{ss-inter}, $\mathcal{S}_X$ is closed under finite intersections. Also, since all the elements of $\mathcal{S}_X$ are finite dimensional $\cl_0$-closed sets, any intersection of elements of $\mathcal{S}_X$ is the intersection of finitely many of them. Thus, the intersection of all elements of $\mathcal{S}_X$ is in $\mathcal{S}_X$, and it is indeed the strong closure of $X$.

For arbitrary $X$, we have
\[
\scl(X) = \bigcup \{ \scl(X') : X' \subsetfin X \}.
\]
To see this note that the set on the right hand side is strong in $A$ because  it is the union of a directed system of finite-dimensional $\cl_0$-closed strong subsets of $A$; also, it is contained in any strong subset of $A$ that contains $X$, for it is clear from the definitions that such a set must contain each of the $\scl(X')$ for $X' \subsetfin X$.
\end{proof}

Let us finish this subsection introducing some further conventions.

\begin{Def}
Let $\X,\Y$ be structures in $\Sub\CC$. 

If $\X$ is a substructure of $\Y$ and $X$ is a strong subset of $\Y$, then we say that $\Y$ is a \emph{strong extension} of $\X$ and write $\X \str \Y$.

An embedding $f:\X \to \Y$ is said to be a \emph{strong embedding}, if $f(X)$ is a strong subset of $Y$. If $Z \subset X,Y$, we say that $f$ is an embedding \emph{over $Z$} if $f$ fixes $Z$ pointwise.
\end{Def}

\subsection{The class $\CC_0$}

Let us fix, for the rest of the paper, a structure $\X_0 = (X_0,G_0)$ in $\Sub \CC$ of finite $\cl_0$-dimension .

Let $\CC_0$ be the class of structures $\A$ in $\CC$ into which there is a strong embedding of $\X_0$, expanded by constants for the image of such an embedding. For $\A \in \CC_0$, we identify $\X_0$ with the strong substructure of $A$ consisting of the interpretations of the constants; with this convention in place, we shall not be explicit about the interpretation of the constants in our notation for the structures in $\CC_0$.

Let $\Sub\CC_0$ be the class of substructures of structures in $\CC_0$ whose domain is a $\cl_0$-closed set (necessarily containing $X_0$). Also, let $\Fin\CC$ be the class of structures in $\Sub\CC$ whose domain has finite $\cl_0$-dimension (over $X_0$).

It is clear from the definitions that $\X_0$ embeds strongly in every structure in $\Sub\CC_0$, therefore we say that it is \emph{prime} in $\Sub\CC_0$ with respect to strong embeddings. Similarly, setting $A_0:= \acl_{T_{\Alg}}(X_0)$ and $\A_0 = (A_0,G_0)$, the structure $\A_0$ is in $\CC_0$ and embeds strongly in all structures in $\CC_0$. Hence we say that $\A_0$ is \emph{prime} in $\CC_0$ with respect to strong embeddings. 

Also, note that for all $\A=(A,G) \in \CC_0$, we have $\Tor G = \Tor G_0$; indeed, recall that $A = \Alg(K)$ for some algebraically closed field $K \supset k_0$ and note that all the torsion points of $A$ have coordinates in $k_0\alg$, therefore we have $\Tor A = \Tor \Alg(k_0\alg) = \Tor A_0$, and hence also $\Tor G  = \Tor A \cap G = \Tor A_0 \cap G = \Tor G_0$.


It is easy to see that for structures in $\CC_0$ (or $\Sub\CC_0$), the localisation $\delta_{X_0}$ is a submodular predimension function with respect to the modular pregeometry $(\cl_0)_{X_0}$.

Henceforth, we shall always work over $X_0$ and, in order to ease the notation, we shall write simply $\delta$ for $\delta_{X_0}$ and $\cl_0$ for $(\cl_0)_{X_0}$. 


The following is a simple but useful remark.

\begin{Rem} \label{ss-checkgreen}
Let $Y$ be a finite dimensional $\cl_0$-closed subset of a structure $\A \in \CC_0$. Recall that, by definition, $Y$ is strong in $\A$ if for every finite dimensional $\cl_0$-closed subset $X$ of $A$, $\delta(X/Y) \geq 0$. We now note that $Y$ is strong in $\A$ if and only if for every finite dimensional $\cl_0$-closed set $X$ \emph{contained in $G$}, $\delta(X/Y) \geq 0$. This follows immediately from the inequality: $\delta(X/Y) \geq \delta((X+Y)\cap G/Y)$. Furthermore, notice that $Y$ is strong in $\A$ if and only if for every \emph{$\End(\Alg)$-linearly independent} tuple $x \subset G$, $\delta(x/Y) \geq 0$.
\end{Rem}


\subsection{Rich structures} \label{sec:existence}


\begin{Def}
A structure $\A \in \CC_0$ is said to be \emph{rich} if for every $\X,\Y \in \Fin\CC_0$ with $\X \str \A$ and $\X \str \Y$, there exists a strong embedding of $\Y$ into $\A$ over $\X$.
\end{Def}

Let us recall that a collection $\F$ of partial isomorphisms from a structure $\A_1$ to a structure $\A_2$ is said to be a \emph{back-and-forth system} if $\F$ is non-empty and has the following properties:
\begin{itemize}
\item \emph{(Forth)} For all $f \in \F$, for all $a_1 \in A_1$, there exists $g \in \F$ extending $f$ such that $a_1$ is in the domain of $g$. 
\item \emph{(Back)} For all $f \in \F$, for all $a_2 \in A_1$, there exists $g \in \F$ extending $f$ such that $a_2$ is in the image of $g$. 
\end{itemize}
If there exists a back-and-forth system of partial isomorphisms from $\A_1$ to $\A_2$, then the structures $\A_1$ and $\A_2$ are said to be \emph{back-and-forth equivalent}. It is a theorem of Karp that two structures are back-and-forth equivalent if and only if they satisfy the same $L_{\infty \omega}$-sentences (recall that $L_{\infty \omega}$ is the extension of first-order logic where conjunctions of arbitrary sets of formulas in a given finite set of variables are allowed as formulas). From this, it follows that every element of a back-and-forth system is a partial elementary map and that back-and-forth equivalent structures are elementarily equivalent.

\begin{Rem} \label{RemF}
Notice that a structure $\A \in \CC_0$ is is rich if and only if for every $\X,\Y \in \Fin\CC_0$ and every strong embedding $f$ of $\X$ into $\Y$, there exits a strong embedding of $\Y$ into $\A$ extending $f^{-1}$. Therefore  if $\A_1$ and $\A_2$ are rich structures in $\CC_0$, then the set of partial isomorphisms 
\[
\F(\A_1,\A_2) = \{ f:X_1 \xrightarrow{\cong} X_2 : \text{ $X_i \str A_i$, $X_i$ fin. dim. $\cl_0$-closed}, i=1,2\}.
\]
is a back-and-forth system. Thus, all rich structures in $\CC_0$ are $L_{\infty \omega}$-equivalent and, in particular, elementarily equivalent.
\end{Rem}

We further note the following:

\begin{Lemm} \label{bnf}
Let $\A_1$ and $\A_2$ be rich structures in $\CC_0$. Let $f: X_1\xrightarrow{\cong} X_2$ be a partial isomorphism from $\A_1$ to $\A_2$ such that $X_i$ is a strong $\cl_0$-closed subset of $A_i$, for $i=1,2$, respectively. Then $f$ is an elementary map. 
\end{Lemm}
\begin{proof}
If $X_1$ and $X_2$ are finite dimensional, then $f$ is in the back-and-forth system $\F(\A_1,\A_2)$ and is hence an elementary map. 

In general, $X_1$ is the union of a directed system of finite dimensional sets $(X_1^i)$, strong in $A$. Note that $X_2$ is the union of the directed system of sets $f(X_1^i)$. Since each $X_1^i$ is strong in $\A_1$, $X_1^i$ is, in particular, strong in $X_1$. Since $f$ is a partial isomorphism, each $f(X_1^i)$ is therefore strong in $X_2$. Therefore, by transitivity, each $f(X_1^i)$ is strong in $\A_2$. Thus, every restriction of $f$ to an $X_1^i$ is a partial isomorphism between a strong subset of $\A_1$ and a strong subset of $\A_2$ and hence is an elementary map. Since $f$ is the union of the directed system of its restrictions to the $X_1^i$, it follows that $f$ is an elementary map.
\end{proof}


The following lemma gives sufficient conditions for the existence of rich structures in the class $\CC_0$.

\begin{Lemm} \label{rich-exist}
Assume the following:
\begin{itemize}
\item \emph{($\Sub\CC_0$ has the amalgamation property for strong embeddings)} For all $\Y_0, \Y_1, Y_2 \in \Sub\CC_0$ with $\Y_0 \str \Y_1$ and $\Y_0 \str \Y_2$, there exist $\Y \in \Sub\CC_0$ and strong embeddings $j_1: \Y_1 \to \Y$ and $j_2:\Y_2 \to \Y$ with $j_1|_{\Y_0} = j_2|_{\Y_0}$.
\item \emph{($\CC_0$ is closed under unions of strong increasing chains)} If $(\A_i)_{i \in I}$ is a strong increasing chain of structures in $\CC_0$, then the structure $\A = \bigcup_{i \in I} \A_i$ is in $\CC_0$.
\item \emph{(Extension property)} for all $\Y \in \Sub\CC_0$ there exists $\A \in \CC_0$ with $\Y \str \A$.
\end{itemize}
Then for every $\Y \in \Sub \CC_0$ there exists a rich structure $\A \in \CC_0$ with $\Y \str \A$.
\end{Lemm}
\begin{proof}
It is sufficient to show the following: 

\noindent\textbf{Claim:} for all $\Y \in \Sub \CC_0$, there exists $\A \in \CC_0$ with $\Y \str \A$ such that for all $\X,\X' \in \Fin\CC_0$ with $\X \str \Y$ and $\X \str \X'$, there is a strong embedding of $\X'$ into $\A$ over $\X$. 

Indeed, the above claim implies the lemma: given any $\Y \in \Sub \CC_0$, we can construct a strong increasing chain of structures $(\A_i)_{i<\omega}$ in $\CC_0$ starting with any $\A_0$ in $\CC_0$ with $\Y \str \A_0$ (such exist by the extension property), and inductively taking $\A_{i+1}$ to be as provided by the claim for $\A_i$. Since $\CC_0$ is closed under unions of strong increasing chains, the structure $\A = \bigcup_{i} \A_i$ is in $\CC_0$. It is easy to see that $\A$ is rich and $\Y \str \A$.

\noindent\textbf{Proof of the claim:} Let $\Y \in \Sub\CC_0$. Let $((\X_i,\X'_i))_{i<\lambda}$ be an enumeration of all pairs $(X,X')$ with $\X,\X' \in\Fin\CC_0$ , $\X \str \Y$ and $\X \str \X'$. We define a strong increasing chain $(\A_i)_{i <\lambda}$ of structures in $\CC_0$ as follows: Let $\A_0$ be any structure in $\CC_0$ with $\Y \str \A_0$ (extension property). For each limit ordinal $0< \sigma <\lambda$, let $\A_\sigma := \bigcup_{i <\sigma} \A_i$. For each $i<\lambda$, let $\ZZ \in \Sub\CC_0$ be as provided by the amalgamation property of $\Sub\CC_0$ for the extensions $\X_i \str \A_i$, $\X_i \str \X'_i$; after an identification, we may assume $\A_i \str \ZZ$; let $\A_{i+1}$ be a structure in $\CC_0$ with $\ZZ \str \A_{i+1}$ (extension property). The structure $\A := \bigcup_{i<\lambda} \A_i$ is then as required.
\end{proof}


We now prooced to show that the sufficient conditions found in \ref{rich-exist} are satisfied. 

\subsubsection{Amalgamation property}

\begin{Def}[Free Amalgam]
Let $\X_i = (X_i,G_i)$, for $i = 1,2,3$, be structures in $\Sub\CC$ and assume $\X_1 \subset \X_i$ for $i=2,3$. The \emph{free amalgam} $\X = (X,G)$ of $\X_2$ and $\X_3$ over $\X_1$ is defined as follows:
Replace $\X_3$ by an isomorphic copy over $\X_1$ so that $X_2$ and $X_3$ are $\acl_{T_{\Alg}}$-independent over $X_1$ inside some model $\bar A$ of $T_{\Alg}$.
Then let $X := X_2 + X_3$ (in $\bar A$) with the induced structure from $\bar A$ and let $G = G_2 + G_3$. 
\end{Def}

\begin{Rem}
Note that the above $\X$ is in $\Sub\CC$. Indeed, $A := \acl_{T_{\Alg}}(X)$ is an infinite algebraically closed set in a model of the strongly minimal theory $T_{\Alg}$ and hence a model of $T_{\Alg}$. Also, it is easy to see that $G$ is a divisible subgroup of $X$.

See that $\X$ comes with canonical embeddings $\X_i \to \X$ over $\X_1$, $i=2,3$. Moreover, $\X$ is unique up to isomorphism over these embeddings. 
\end{Rem}

\begin{Lemm}[Asymmetric Amalgamation Lemma] \label{aap-cc0}
Let $\X_i = (X_i,G_i)$, for $i = 1,2,3$, be structures in $\Sub\CC_0$ and assume $\X_1 \subset \X_i$ for $i=2,3$. If $\X_1$ is strong in $\X_2$, then the free amalgam $\X$ of $\X_2$ and $\X_3$ over $\X_1$ is in $\Sub\CC_0$ and, under the canonical embedding, $\X_3$ is strong in $\X$.
\end{Lemm}
\begin{proof}
Let us identify $\X_3$ with its image under the canonical embedding into $\X$. Since we know $\X \in \Sub\CC$, by transitivity it suffices to show that $\X_3$ is strong in $\X$. Let $Y$ be a finite dimensional $\cl_0$-closed subset of $G$ and let us show that $\delta(Y/X_3) \geq 0$ (this is enough, by \ref{ss-checkgreen}). Since $G = G_2 + G_3$, there exists finite dimensional $\cl_0$-closed sets $Y_2$ and $Y_3$ of $X_2$ and $X_3$, respectively, such that $Y = Y_2 + Y_3$. Note that $\delta(Y/X_3) = \delta(Y_2/X_3)$. Also, by the independence of $X_2$ and $X_3$ over $X_1$, we know that $\delta(Y_2/X_3) = \delta(Y_2/X_1)$. Thus, $\delta(Y/X_3) = \delta(Y_2/X_1)$, and the latter is non-negative because $\X_1$ is strong in $\X_2$.
\end{proof}

\begin{Cor} \label{ap-cc0}
$\Sub\CC_0$ has the amalgamation property with respect to strong embeddings.
\end{Cor}

\begin{Rem}
Since the free amalgam of finite dimensional structures in $\Sub\CC_0$ is finite dimensional, we also have that the amalgamation property with respect to strong embeddings holds for the class $\Fin\CC_0$ of finite dimensional structures in $\Sub\CC_0$.
\end{Rem}


\subsubsection{Unions of chains}

\begin{Lemm} \label{uc-cc0}
The class $\CC_0$ is closed under unions of strong increasing chains.
\end{Lemm}
\begin{proof}
Since $T_{\Alg}$ and the theory of divisible abelian groups are $\forall \exists$-axiomatizable, so is the class $\CC$. Therefore $\CC$ is closed under unions of increasing chains. Since every element of a strong increasing chain is strong in the union of the chain and we have transitivity of strong extensions (Lemma~\ref{ss-props}), it follows immediately that the class $\CC_0$ is closed under unions of strong increasing chains.
\end{proof}

\subsubsection{Extension property}

\begin{Lemm} \label{ext-cc0}
For all $\X$ in $\Sub\CC_0$ there exists $\A \in \CC_0$ with $X \str \A$.
\end{Lemm}
\begin{proof}
Let $\X = (X,G)$ be a structure in $\Sub\CC$. Let $\bar A$ be a model of $T_{\Alg}$ with $X \subset \bar A$. In $\bar A$, let $A = \acl_{T_{\Alg}}(X)$. Note that $A$, with the induced structure from $\bar A$, is a model of $T_{\Alg}$. Then $\A = (A,G)$ is a structure in $\CC$. It is easy to see that, not having any new green points, $\A$ is a strong extension of $\X$. By transitivity of strong extensions, if $\X$ is in $\Sub\CC_0$, then $\A$ is in $\CC_0$.
\end{proof}

Also, we have the following useful fact:

\begin{Lemm} \label{lemm-ss-acl}
Let $\A$ be a structure in $\CC$. Let $Y$ be a finite dimensional $\cl_0$-closed subset of $A$ and let $Z = \acl_{T_{\Alg}}(Y)$. Then $Y$ is strong in $\A$ if and only if $Z$ is strong in $\A$.
\end{Lemm}
\begin{proof}
Assume $Y$ is strong in $A$. Let $Y'$ be a finite dimensional $\cl_0$-closed subset of $Z$ containing $Y$. Let $X$ be a finite dimensional $\cl_0$-closed subset of $A$. Since $\acl_{T_{\Alg}}(Y) = \acl_{T_{\Alg}}(Y')$, we have $\trd(X/Y') = \trd(X+Y'/Y)$. Since $Y$ is strong, there is no green point in $Z \setminus Y$ and hence $Y \cap G = Y' \cap G = Z \cap G$; therefore $\ld_G(X/Y') = \ld((X+Y')\cap G/Y'\cap G) = \ld((X+Y')\cap G/Y\cap G) = \ld_G(X+Y'/Y)$. It follows that $\delta_G(X/Y') = \delta_G(X+Y'/Y) \geq 0$. Thus, $Y'$ is strong in $\A$. Since $Z$ is the union of the directed system of all such $Y'$, we get that $Z$ is strong in $A$.

To prove the converse, suppose $Z$ is strong in $A$. Then, by definition, it is the union of a directed system of strong finite dimensional $\cl_0$-closed sets. Since $Y$ is finite dimensional, using the directedness of the system we can find a finite dimensional $\cl_0$-closed subset $Y'$ of $Z$ that contains $Y$ and is strong in $A$. For every finite dimensional $\cl_0$-closed subset $X$ of $A$, we have $\trd(X/Y) = \trd(X/Y')$ and $\ld_G(X/Y) \leq \ld(X/Y')$, hence $\delta_G(X/Y) \geq \delta_G(X/Y') \geq 0$. Therefore $Y$ is strong in $A$.
\end{proof}

\subsection{The associated dimension function}


We close this section introducing the dimension function $\dd$ associated to the predimension function $\delta$. This will be a important tool in the model-theoretic analysis in Section~\ref{sec:mth} of the theories constructed in Section~\ref{sec:theory}.

\begin{Def} \label{pred2dim}
For any $\A \in \CC_0$, the \emph{dimension function $\dd$ associated to $\delta$} is defined for all finite $X \subset A$ by the formula
\[
\dd(X) = \min \{ \delta(X') : X \subset X' \subsetfin A\}.
\]
\end{Def}

\begin{Rem} \label{pred2cl}
The function $\dd$ has the following properties: 
\begin{itemize}
\item $\dd(\emptyset) = 0$.
\item For all $X,Y \subsetfin A$, if $X \subset Y$ then $\dd(X) \leq \dd(Y)$.
\item For all $X,Y,Z \subsetfin A$, if $\dd(XY) = \dd(Y)$ then $\dd(XYZ) = \dd(YZ)$.
\footnote{Equivalently, for all $Y,Z \subsetfin A$ and all $x \in A$, if $\dd(xY) = \dd(Y)$ then $\dd(xYZ) = \dd(YZ)$.}
\end{itemize}
It follows that associated to $\dd$ we have a closure operator with finite character $\cl_{\dd}$ on $A$ which restricts to a pregeometry on the set $A_1 := \{ x \in A : \dd(x) \leq 1 \}$ with dimension function $\dd$. Indeed, the operator $\cl_{\dd}$ is given by: for $X \subsetfin A$ and $x_0 \in A$, 
\begin{align*}
     &x_0 \in \cl_{\dd}(X)\\
\iff &\dd(x_0/X) = 0\\
\iff &\dd(x_0 X) = \dd(X)\\
\iff &\text{ There exists a tuple } x \supset x_0 \text{ such that } \delta(x/\scl(X)) = 0\\
\iff &x_0 \in \scl(X) \text{ or } \text{there exists a tuple } x \supset x_0 \text{, $\cl_0$-independent over $\scl(X)$,}\\ &\text{such that } \delta(x/\scl(X)) = 0.   
\end{align*}
\end{Rem}

\begin{Rem}
Note that if $x$ is strong in $\A$, then $\delta(x) = \dd(x)$. Also, for all $x$, $\dd(x) = \delta(\scl(x))$.
\end{Rem}


\section{Theories} \label{sec:theory}

We now turn to the task of finding axioms for the complete theory common to all rich structures in the class $\CC_0$.

We shall henceforth assume that our choice of $\X_0$ is such that $X_0$ has a $\cl_0$-basis consisting of green points.


\subsection{Axiomatizing $\CC_0$} \label{defty}

The first step in finding axioms for the theory of rich structures in $\CC_0$ is to axiomatize the class $\CC_0$. 

\subsubsection{Intersections of subvarieties with subgroups}

Let $B = \mathbb{B}(K)$, with $K$ algebraically closed, be a smooth algebraic variety. If $V,W$ are subvarieties of $B$ such that $V \cap W \neq \emptyset$, then every irreducible component of the intersection has dimension at least $\dim V + \dim W - \dim B$. This follows from \cite[I.6 Theorem 6]{Shaf1} and the fact that dimension is a local notion. The following definition comes from \cite{ZCIT} (Definitions 2 and 3).

\begin{Def}
Let $V,W$ be subvarieties of $B$ with non-empty intersection. Let $S$ be an irreducible component of $V \cap W$. Then, $S$ is said to be an \emph{atypical component of the intersection of $V$ and $W$} if
\[
\dim S > \dim V + \dim W - \dim B.
\]
Otherwise, that is if $\dim S = \dim V + \dim W - \dim B$, we say that $S$ is a \emph{typical component of the intersection of $V$ and $W$}. 

Let $B'$ be a smooth subvariety of $B$ containing $S$. Then, $S$ is said to be an \emph{atypical component of the intersection of $V$ and $W$ with respect to $B'$} if
\[
\dim S > \dim V\cap B' + \dim W\cap B' - \dim B'.
\]
Otherwise, that is if $\dim S = \dim V\cap B' + \dim W\cap B' - \dim B'$, we say that $S$ is a \emph{typical component of the intersection of $V$ and $W$ with respect to $B'$}.
\end{Def}

In order to explain the terminology introduced in the above definition, let us note that if the dimension of the intersection of $V$ and $W$ is larger than $\dim V + \dim W - \dim B$, this is due to the existence of algebraic dependences between the equations defining $V$ and those defining $W$. Such dependences are closed conditions, therefore in generic cases the dimension of each of the irreducible components of the intersection is equal to $\dim V + \dim W - \dim B$.

We shall now state Zilber's Conjecture on Intersections with Tori (CIT). We formulate the statement in the full generality of semiabelian varieties. In fact, we will not work at the level of generality of semiabelian varieties but only in the cases of powers of an elliptic curve and of algebraic tori. Thus, for our purposes it suffices to know that both abelian varieties and algebraic tori are semiabelian varieties.

In the following statements (\ref{cit}-\ref{m-l}), $K$ will always denote an algebraically closed field extending $k_0$.

\begin{Conj}[CIT] \label{cit}
Let $B=\mathbb{B}(K)$ be a semiabelian variety defined over a field $k_0$ of characteristic zero. 

For every $k\geq 0$, every subvariety $W$ of $B$ defined over $k_0$, there exists a finite collection of proper algebraic subgroups $C_1,\dots,C_s$ of $B$ with the following property: if $S$ is an atypical component of the intersection of $W$ and a proper algebraic subgroup $C$ of $B$, then for some $i \in \{1,\dots,s\}$, $S$ is contained in $C_i$. 
\end{Conj}

In the multiplicative group case, the above is Conjecture~1 in \cite{ZCIT}. There it is shown that the CIT implies the following version that allows parameters (see Theorem 1 and Proposition 1 in that paper).

\begin{Conj}[CIT with parameters] \label{cit-param}
Let $B=\mathbb{B}(K)$ be a semiabelian variety defined over a field $k_0$ of characteristic zero. 

For every $k\geq 0$, every subvariety $W(x,y)$ of $B^{1+k}$ defined over $k_0$ and every $c \in B^k$, there exists a finite collection of proper algebraic subgroups $C_1,\dots,C_s$ of $B$ and elements $\alpha^1,\dots,\alpha^r$ of $B$  with the following property: for every coset $\alpha + C$ of a proper algebraic subgroup $C$ of $B$, if $S$ is an atypical component of the intersection of $W(x,c)$ and $\alpha + C$, then for some $i \in \{1,\dots,s\}$ and some $j \in \{1,\dots,r\}$, $S$ is contained in $\alpha^j + C_i$ and $S$ is a typical component of the intersection of $W(x,c)$ and $\alpha + C$ with respect to $\alpha^j + C_i$.
\end{Conj}

The following theorem deals with the same situation as the CIT. Here, however, the conclusion is weaker. Following common practice, we refer to the theorem as \emph{Weak CIT}. For the multiplicative group this is Corollary~3 in \cite{ZCIT} and Corollaire 3.6 in \cite{PzEq3}. In the general case of semiabelian varieties the result is due to Kirby, Theorem 4.6 in \cite{KirSemiab}.

\begin{Thm}[Weak CIT] \label{weakcit}
Let $B=\mathbb{B}(K)$ be a semiabelian variety defined over a field $k_0$ of characteristic zero. 

Let $k\geq 0$ and let $W(x,y)$ be a subvariety of $B^{1+k}$ defined over $k_0$. 

Then there exists a finite collection of proper algebraic subgroups $C_1,\dots,C_s$ of $B$ with the following property: for any $c \in B^k$ and any coset $\alpha + C$ of a proper algebraic subgroup $C$ of $B$, if $S$ is an atypical component of the intersection of $W(x,c)$ and $\alpha + C$, then for some $i \in \{1,\dots,s\}$ and some $\alpha'\in B$, $S$ is contained in $\alpha' + C_i$ and $S$ is a typical component of the intersection of $W(x,c)$ and $\alpha + C$ with respect to $\alpha' + C_i$.
\end{Thm}
 
Indeed, it is easy to see that the CIT with parameters (and hence also the CIT) implies the Weak CIT. The difference between the two consists in that the CIT with parameters provides a finite collection of cosets of algebraic subgroups that controls all atypical intersections while the Weak CIT only gives a finite collection of subgroups such that the collection of all their cosets controls all atypical intersections.

\subsubsection{The Mordell-Lang property}

We are now interested in the content of the (absolute) Mordell-Lang conjecture, in characteristic zero. This is a theorem, after work of Laurent, Faltings, Vojta, Raynaud and McQuillan. For more precise attributions and bibliography we refer to \cite{HinBous}.

\begin{Thm}[Mordell-Lang conjecture] \label{m-l} 
Let $K$ be an algebraically closed field of characteristic 0. Let $B=\mathbb{B}(K)$ be a semiabelian variety. Let $\Gamma$ be a subgroup of $B$ of finite rank. Then for every subvariety $W$ of $B$, there exist a natural number $r$, elements $\gamma_1,\dots,\gamma_r$ of $\Gamma$ and algebraic subgroups $B_1,\dots,B_r$ of $B$ such that $\gamma_i + B_i \subset W$ and
\[
W \cap \Gamma = \bigcup_{i=1}^r \gamma_i + (B_i \cap \Gamma).
\]
\end{Thm}

As it is already in use, we shall say that a subgroup $\Gamma$ of $B$ has the \emph{Mordell-Lang property} if it satisfies the conclusion of the above theorem.


\subsubsection{The theory $T^0$}

The following lemma is the result for axiomatizing the class $\CC_0$. It generalises Corollaire~3.4 of \cite{PzEq3}.

Given an algebraic variety (or, more generally, a definable set) of the form $W(x,y)$, the algebraic variety (respectively, definable set) $W(x,c)$ is also denoted by $W_c$.

\begin{Lemm} \label{lemm-ss-def}
Let $\A = (A,G) \in \CC$. For every complete $L_{\Alg}$-$l$-type $\Theta(y)$, there exists a partial $L$-$l$-type $\Phi_\Theta(y)$ consisting of universal formulas with the following property: for every realisation $c$ of $\Theta$ in $\A$ with $c \in G^l$,
\[
\text{$\A \vDash \Phi_\Theta(c)$ if and only if $\spank c$ is strong in $\A$.}
\]
\end{Lemm}
\begin{proof}
For any type $\Theta$ with no realisations consisting purely of green points the statement is trivial, thus assume we have a realisation $c' \in G^l$ of $\Theta(y)$. Let $\Phi_\Theta(y)$ be the partial type consisting of the following formulas:

For each $n \geq 1$ and each subvariety $W(x,y)$ of $A^{n+l}$ defined over $k_0$ such that $W_{c'}$ is irreducible over $k_0(c')$ and has dimension $< \frac{n}{2}$, the formula
\begin{multline*}
\forall x \Big( \big( W(x,y) \land \bigwedge_{1 \leq j \leq n} G(x_j) \land \neg W^*(x,y) \big)\\ 
\rightarrow \bigvee_{1\leq i \leq s} \bigvee_{\substack{1 \leq j\leq r_i\\ B_{ij} \text{ proper}}} N^{ij}\cdot y + B_{ij}(n_{ij}(M^i\cdot x)) 
\Big),
\end{multline*}
where: 
\begin{itemize}
\item $s$, $C_1,\dots,C_s$ are as provided by Theorem~\ref{weakcit} (Weak CIT) for the subvariety $W_{c'}$ of $A^n$, and each $C_i$ is defined by the system of equations $M^i \cdot x = 0$, $M^i \in \Mat_{n_i \times n}(\End(\Alg))$ of rank $n_i$;
\item for each $i \in \{1,\dots, s\}$; $r_i$, $\gamma'_{i1},\dots,\gamma'_{ir_i}$, $B_{i1},\dots,B_{ir_i}$ are as provided by Theorem~\ref{m-l} (Mordell-Lang property) for the variety $W_i$, which we define to be the $k_0(c')$-Zariski closure of $M^i \cdot W_{c'}$, and the finite rank subgroup $(\spank c')^{n_i}$ of $A^{n_i}$; and $N^{i1},\dots,N^{ir} \in \Mat_{n_i\times l}(\End(\Alg))$, $n_{i1},\dots, n_{ir_i} \in \N$ are such that $n_{ij} \gamma'_{ij} = N^{ij} \cdot c'$;
\item $W^*(x,y) := \bigcup_{i=1}^s W^{*i}(x,y)$ and, for each $i=1,\dots, s$, we define $W^{*i}(x,y)$ to be a variety such that $W^{*i}(x,c')$ is the $k_0(c')$-Zariski closure of the set
\[
\{ x \in W_{c'} : \dim W_{c'} \cap x + C_i > \dim W_{c'} - \dim W_i \}.
\]
Note that the above set is the union of the non-generic (i.e. not of minimal dimension) fibres inside $W_{c'}$ for the map given by $x \mapsto x^{M^i}$. By a standard fact, this set is contained in a proper closed subset of $W_{c'}$. Therefore $W^*_{c'} \subsetneq W_{c'}$. 
\end{itemize}

Let $c$ be any realisation of $\Theta$ in $\A$ with $c \subset G$. Note that since $\Theta$ is a complete $L_\Alg$-type, $c$ and $c'$ are conjugates by an automorphism of the $L_{\Alg}$-structure $A$.

Suppose $\A \models \Phi_\Theta(c)$. To see that $\spank c$ is then strong in $\A$, suppose towards a contradiction that there exists $b \in A^n$ such that $\delta(b/\spank c) < 0$. It is easy to see that we may assume $b$ to be in $G^n$ and linearly independent over $\spank c$. Let $W_c := W(x,c)$ be the algebraic locus of $b$ over $k_0(c)$. Then, since $\delta(b/\spank c) < 0$, we have $\dim W_c < \frac{n}{2}$. Hence also $W_{c'}$ is irreducible over $k_0(c')$ and $\dim W_{c'} <\frac{n}{2}$. Thus, there is a formula in $\Phi_\Theta(y)$ corresponding to $W(x,y)$. If the disjunction in the formula is non-empty then we get a linear dependence on $b$ over $\spank c$, hence a contradiction. If the disjunction is empty, then the fact that the formula is satisfied by $c$ means that the set $(W_c \setminus W^*_c) \cap G^n$ is empty; but our $b$ is in this set ($b$ is not in $W^*_c$ because it is generic and, as noted earlier, $W^*_c$ is contained in a proper closed subset of $W_c$), hence also a contradiction. This proves that $\spank c$ is then strong in $\A$.

Conversely, assume $\spank c$ is strong in $A$ and let us see that $\A \models \Phi_\Theta(c)$.
Let $n \geq 1$ and let $W(x,y)$ be a subvariety of $A^{n+l}$ over $k_0$ such that $W(x,c)$ is irreducible over $k_0(c)$ and of dimension $< \frac{n}{2}$ and suppose $b$ is an element of the set $(W_c \setminus W^*_c) \cap G^n$.

Since $\trd(b/c) \leq \dim W < \frac{n}{2}$ and by assumption $\delta(b/\spank c) \geq 0$, the tuple $b$ must be linearly dependent over $\spank c$. Thus, let $\alpha + C$ be a coset of a proper algebraic subgroup of $A^n$ containing $b$ of dimension $\ld(b/\spank c)$, $\alpha \in \spank c$.

Let $S$ be an irreducible component of $W_c \cap \alpha + C$ containing $b$. Then $S$ is an atypical of the intersection of $W_c$ and $\alpha + C$: to see this note that, on the one hand, $S$ is defined over $k_0(c)\alg$ and so $\dim S \geq \trd(b/c) \geq \frac{1}{2} \ld(b/\spank c) = \frac{1}{2} \dim C$ and, on the other hand, since $\dim W_c <\frac{n}{2}$, we have $\dim W_c + \dim (\alpha + C) - n < \dim C - \frac{n}{2} < \frac{1}{2} \dim C$.

Applying an automorphism $\sigma \in \Aut(A)$ with $\sigma(c) = c'$, we have that $\sigma(S)$ is an atypical component of the intersection of $W_{c'}$ and $\sigma(\alpha) + C$. Therefore there exists $i \in \{1,\dots,s\}$ such that $\sigma(S)$ is contained in a coset $\sigma(\alpha^*) + C_i$ and $\sigma(S)$ is a typical component of the intersection of $W_{c'}$ and $\sigma(\alpha) + C$ with respect to $\sigma(\alpha^*) + C_i$. Applying $\sigma^{-1}$, we get that $S$ is contained in $\alpha^* + C_i$ and $S$ is a typical component of the intersection of $W_c$ and $\alpha + C$ with respect to $\alpha^* + C_i$.

Since $S$ is defined over $k_0(c)\alg$, it has $k_0(c)\alg$-rational points. Hence we may assume $\alpha^* \in (k_0(c)\alg)^n$. Let us now look at the coefficients of the equations defining the coset $\alpha^* + C_i$, namely $\beta^* := M^i \cdot \alpha^* \in (k_0(c)\alg)^{n_i}$. Also, $\beta^* = M^i \cdot b \in G^{n_i}$. Thus, since $\spank c$ is strong, $\beta^* \in (\spank c)^{n_i}$. So, $M^i \cdot b \in W_i \cap (\spank c)^{n_i}$ and, applying appropriate automorphisms as before, we have $W_i \cap (\spank c)^{n_i} = \bigcup_{j=1}^{r_i} \gamma_{ij} + (B_{ij} \cap (\spank c)^{n_i})$, where $\gamma_{ij} \in \spank c$ satisfies $n_{ij} \gamma_{ij} = N^{ij} \cdot c$. Therefore we can find $j \in \{1,\dots,r_i\}$ such that $M^i \cdot b \in \gamma_{ij} + B_{ij}$, and hence $n_{ij}(M^i\cdot b) \in (N^{ij}\cdot c) + B_{ij}$.

It now suffices to show that $B_{ij}$ is a \emph{proper} algebraic subgroup of $A^n$. This follows from the fact that $W_i$ is a proper subvariety of $A^{n_i}$, as the following dimension calculations show: first, from the atypicality of $S$ we have
\[
\dim S > \dim W_c + \dim C - n.
\]
Also, from the typicality of $S$ with respect to $\alpha^* + C_i$ we have
\[
\dim S = \dim W_c \cap (\alpha^* + C_i) + \dim (\alpha + C) \cap (\alpha^* + C_i) - \dim (\alpha^* + C_i).
\]
Combining the last two expressions we get
\[
\dim W_c + \dim C - n < \dim W_c \cap (\alpha^* + C_i) + \dim (\alpha + C) \cap (\alpha^* + C_i) - \dim (\alpha^* + C_i).
\]
Reorganising terms and noting that $\alpha + C = b + C$ and $\alpha^* + C_i = b + C_i$,
\begin{align*}
&\dim W_c - \dim W_c \cap (b + C_i)\\ <\ &n - \dim (b + C_i) + \dim (b + C) \cap (b + C_i) - \dim (b + C)\\
                                       \leq\  &n - \dim (b + C_i)\\
                                       =\ &n_i.
\end{align*}
Since $b$ is not in $W^*_c$, we know $\dim W_i = \dim W_c - \dim W_c \cap (b + C_i)$. Therefore $\dim W_i <n_i$. 
\end{proof}

\begin{Rem}
If one works under the assumption that the group $G$ is torsion-free, then a simpler argument, using the Weak CIT but not the Mordell-Lang property, suffices to prove the above lemma. Indeed, this is the well-known argument of Poizat in Corollaire~3.4 of \cite{PzEq3}.

In \cite{PzEq3}, it is noted that in the more general situation, where the torsion of $G$ is not necessarily trivial, the statement holds if one assumes the CIT. Without the extra assumption, however, the question of how to get the result was left open. The above proof answers this question.

Let us remark some limitations of the above lemma, in comparison with the argument in \cite{PzEq3}, which applies to the torsion-free case. One limitation is that we had to restrict to tuples $c$ with coordinates in $G$, which is not necessary there. But also, the result there is more uniform, since for each $l$, it gives a type $\Phi_l(y)$ that works for all $l$-tuples $b$, independently of their algebraic types. This difference is due to the fact that the Weak CIT is uniform in families, but the same kind of uniformity is not available for the Mordell-Lang property.
\end{Rem}
%
%
%


\begin{Lemm} \label{lemm-ss-def2}
There exists an $L_{X_0}$-theory $T^0$ such that for every $L_{X_0}$-structure $\A = (A,G)$,
$\A \models T^0$ if and only if $\A$ is in $\CC_0$.
\end{Lemm}

\begin{proof}
It suffices to show that the following conditions on a structure $(A,G)$ can be expressed by a set of $L_{X_0}$-sentences.
\begin{enumerate}
\item\label{T0-1} $A$ is a model of $T_{\Alg}$,
\item\label{T0-2} $G$ is a divisible subgroup of $A$,
\item\label{T0-3} $\qftp^{\A}(X_0) = \qftp^{\X_0}(X_0)$,
\item\label{T0-5} $X_0$ is strong in $\A$,
\end{enumerate}

It is clear that we can find a set of $L_{X_0}$-sentences $\Sigma$ expressing conditions \ref{T0-1}, \ref{T0-2}, \ref{T0-3}.

Let $c^0$ be a $k_\Alg$-linear basis of $X_0$ consisting of green points, let $\Theta = \qftp_{L_{\Alg}}(c^0)$. By Lemma~\ref{lemm-ss-def}, the set of $L_{X_0}$-sentences $\Phi_\Theta(c^0)$ expresses \ref{T0-5} modulo $\Sigma$. Thus, $T^0 := \Sigma \cup \Phi_\Theta(c^0)$ is as required. 
\end{proof}

Henceforth let $T^0$ denote the theory found in the proof of the above lemma.

\subsection{The theory of the rich structures}

\subsubsection{Rotund varieties}

The rest of Section~\ref{sec:theory} is dedicated to finding a theory whose $\omega$-saturated models are precisely the rich structures in $\CC_0$. We then show that the theory is the complete theory of every rich structure in $\CC_0$.

We start by defining \emph{rotund varieties}, which serve as the main technical tool in finding the required theory.

Let $A = \Alg(K)$ be a model of $T_{\Alg}$.

\begin{Def} \label{def-rotund}
An irreducible subvariety $W$ of $A^n$ is said to be \emph{rotund} if for every $k \times n$-matrix $M$ with entries in $\End(\Alg)$ of rank $k$, the dimension of the constructible set $M \cdot W$ is at least $\frac{k}{2}$.
\footnote{Recall that the dimension of a constructible subset of $A^n$ is, by definition, the dimension of its Zariski closure.}
\end{Def}

\begin{Rem}
For any subvariety $W$ of $A^n$ and any $C\subset A$ such that $W$ is defined over $k_0(C)$, if $b$ is a generic point of $W$ over $k_0(C)$, then: $W$ is rotund if and only if for every $k \times n$-matrix $M$ with entries in $\End(\Alg)$ of rank $k$,
\[
\trd(M \cdot b/C) \geq \frac{k}{2}.
\]
\end{Rem}

\begin{Rem}
If $W$ is a rotund subvariety of $A^n$, then, in particular, for every non-zero $m \in \End({\Alg})^n$, $\dim m \cdot W \geq 1$. This implies that $W$ is not contained in any coset of a proper algebraic subgroup of $A^n$. To refer to this property, we say that $W$ is \emph{free (of linear dependences)}.
\end{Rem}

\begin{Rem} \label{isrotund}
Let us note that rotund varieties correspond to strong extensions of structures in $\Sub \CC_0$ as follows:

Consider a structure $\X \in \Sub\CC_0$, with $X \subset \bar A \models T_{\Alg}$. Let $W$ be an irreducible subvariety of $\bar A^n$ defined over $X$ and let $b$ be a generic point of $W$ over $k_0(X)$ in $\bar A$. Let $Y$ be the substructure of $\bar A$ with domain $Y = X + \spank b$ and set $G(Y):= G(X) + \langle b^i : i \geq 1 \rangle$, where $b^i$ is a sequence of tuples in $(\spank b)^n$ with $b^1 = b$ and for all $i,j \geq 1$, $j b^{ij} = b^i$. Then $\Y = (Y,G(Y))$ is a structure in $\Sub\CC$ extending $\X$. 

Moreover, for every $k \times n$-matrix $M$ of rank $k$ with entries in $\End(\Alg)$, we have
\[
2 \dim (M \cdot W) - k = 2 \trd(M \cdot b/k_0(X)) - \ld(M \cdot b) = \delta(M \cdot b/X).
\]
Hence, if $\Y$ is a strong extension of $\X$, then the above value is always non-negative, and hence $W$ is a rotund variety.

Conversely, assume $W$ is a rotund variety and let us see that then $\Y$ is a strong extension of $\X$. Indeed, for every tuple $b' \subset Y$, there exists a $k \times n$-matrix $M$ such that $\spank(b'X) = \spank((M \cdot b)X)$; therefore $\delta(b'/X) = \delta(M \cdot b/X) = 2 \dim (M \cdot W) - k$ and, by the rotundity of $W$, this value is non-negative.
\end{Rem}


\noindent\textbf{Examples of rotund varieties.}
We shall now give some examples of rotund varieties. Besides illustrating the notion, they will be useful for some of our later arguments. For the first two families of examples let us consider the case where $\Alg$ is the multiplicative group. We thus write the group operation on $A = K^*$ multiplicatively.

\noindent\textit{First example: ``$X+Y = c$''.} For any $c \in A = K^*$, the subvariety of $(K^*)^2$ defined by the equation $x+y = c$, where $+$ denotes addition in $K$, is rotund. This follows from transcendence degree calculations.
\footnote{ Explanation: Note that it is sufficient to show that the variety is free. Let $b = (b_1,b_2)$ be a generic point of the variety defined by $x+y=c$ over $\Q(c)\alg$. Note that $b_1$ is transcendental over $\Q(c)\alg$ and $b_2 = c -b_1$.

Let $m = (m_1,m_2) \in \Z^2$ be non-zero. Suppose towards a contradiction that $b^m = c'$ for some $c'$ in $\Q(c)\alg$. Then $c' = b_1^{m_1} b_2^{m_2} = b_1^{m_1} (c-b_1)^{m_2}$. Since $b_1$ is transcendental over $c$, we see that $m_2 = -m_1$. Thus, $\frac{1}{c'} = \frac{(c-b_1)^{m_1}}{b_1^{m_1}} = (\frac{c}{b_1} - 1)^{m_1}$. But this contradicts that $b_1$ is transcendental over $c$. 
}

\noindent\textit{Second example: Generic hyperplanes.} A \emph{hyperplane} in $K^n$ is a variety defined by an equation of the form $c \cdot x = d$ for some non-zero $c \in K^n$ and some $d \in K$. If $C$ is a subset of $K$ and $c \in K^n$ is such that $\trd(c/C) = n$ then the hyperplane defined by the equation $c \cdot x = 1$ is said to be a \emph{generic hyperplane over $C$}.

Let $H_{n,k}(x,y^1,\dots,y^k)$ be the subvariety of $K^{n+kn}$ defined by the system of equations $M \cdot x = 1$, where $M$ is the $k \times n$-matrix with rows $y^1,\dots,y^k$ and $1$ denotes the tuple in $K^k$ whose entries are all equal to $1$. 

A variety of the form $H_{n,k}(x,c^1,\dots,c^k)$ for some $c^1,\dots,c^k \in K^n$ is the intersection of $k$ hyperplanes. If $C$ is a subset of $K$ and $\trd(c^1,\dots,c^k/C) = nk$ then we say that $H_{n,k}(x,c^1,\dots,c^k)$ is the \emph{intersection of $k$ (independent) generic hyperplanes over $C$}.

The following lemma follows Lemme 3.1 in \cite{PzEq3} and Lemma 5.2 in \cite{ZBicol}.

\begin{Lemm}
If $V \subset (K^*)^n$ is a rotund variety defined over $C \subset K$ of dimension $d$ with $d-1 \geq \frac{n}{2}$ and $H_{n,1}(x,c)$ is a generic hyperplane over $C$, then $V \cap H_{n,1}(x,c)$ is a rotund variety of dimension $d-1$. 
\end{Lemm}
\begin{proof}
Let $H_c$ denote the hyperplane $H_{n,1}(x,c)$.

Let us first show that all irreducible components of $V \cap H_c$ have dimension $d-1$: Since $V$ is defined over $C$, $V$ has rational points in $\Q(C)\alg$ and such points cannot be in $H_c$, for $c$ is assumed to be algebraically independent over $C$, hence $V \not\subset H_c$. Therefore $V \cap H_c$ is a proper subvariety of the irreducible variety $V$. Thus, $\dim V \cap H_c < \dim V = d$. But also, by the smoothness of $(K^*)^n$ and the dimension of intersection inequality, the dimension of every irreducible component of $V \cap H_c$ is at least $\dim V + \dim H_c - n = d + (n-1) - n = d-1$. Thus, every component has dimension $d-1$.

Let us now see that $V \cap H_c$ is in fact irreducible: Let $V_1$ and $V_2$ be irreducible components of $V \cap H_c$ (not necessarily distinct). Let $a^1$ be a generic point of $V_1$ over $Cc$ and let $a^2$ be a generic point of $V_2$ over $Cca^1$. Note $\trd(a^1/Cc) = \trd(a^2/Cca^1) = d-1$. Using the additivity of the transcendence degree, we therefore obtain $\trd(a^1 a^2 c/C) = n +2d -2$. By the independence of $a^1$ and $a^2$ over $C$, $\dim H_{a^1} \cap H_{a^2} = n-2$; hence $\trd(c/C a^1 a^2) \leq n-2$. Using again the additivity we get $\trd(a^1 a^2/C) \geq 2d$. It easily follows that, in fact, $\trd(a^1 a^2/C) = 2d$ and $\trd(c/C a^1 a^2) = n-2$. Thus, $a^1$ and $a^2$ are independent generic points of the irreducible variety $V$ over $C$ and $c$ is a generic point of the irreducible variety $H_{a^1} \cap H_{a^2}$ over $C a^1 a^2$. 
This means that independently of the choice of the components $V_1$ and $V_2$ the type of $a^1 a^2 c$ over $C$ is always the same. Then $V \cap H_c$ must have only one irreducible component, as otherwise we get different types by choosing $V_1 = V_2$ and $V_1 \neq V_2$.

We now turn to showing that $V \cap H_c$ is rotund. Let $a$ be a generic point of $V \cap H_c$ over $\Q(Cc)\alg$. Note that $a$ is then also a generic point of $V$ over $\Q(C)\alg$ and $\trd(c/Ca) = n-1$.

\noindent\textbf{Claim:} Let $b \subset \Q(Ca)\alg$. If $\trd(b/Cc) < \trd(b/C)$, then $a \subset \Q(b)\alg$.

\noindent\textbf{Proof of claim:} Assume $\trd(b/Cc) < \trd(b/C)$. Let $U$ be the locus of $c$ over $\Q(Cb)\alg$. From our assumption, using the exchange principle, we have  $\trd(c/Cb) < \trd(c/C)$. Thus, $\dim U = \trd(c/Cb) < \trd(c/C) = n$.

The hyperplane $H_a$ must be contained in $U$; otherwise, $\dim U \cap H_a < \dim U < n$, contradicting the fact that $\trd(c/Ca) = n-1$. Since $\dim U \leq n-1 = \dim H_a$, we conclude that $U = H_a$.

Since $U$ is defined over $\Q(Cb)\alg$, we can find $b' \in U$ with coordinates in $\Q(Cb)\alg$. Then $a$ is uniquely determined by the conditions (1) $b' \cdot a = 1$ and (2) for all $z \in H_a$, $(z-b')\cdot a = 0$. Thus, $a \subset \Q(Cb)\alg$.

Applying the claim: Let $M$ be a $k \times n$-matrix with integer entries of rank $k$. Let $b = a^M  \in (K^*)^k$. Suppose towards a contradiction that $\trd(b/Cc)< \frac{k}{2}$. By the rotundity of $V$, $\trd(b/C) \geq \frac{k}{2}$. Hence $\trd(b/Cc) < \trd(b/C)$. Therefore, by the claim, $a \subset \Q(Cb)\alg$. Thus, $\trd(b/Cc) = \trd(a/Cc) = d-1 \geq \frac{n}{2} \geq \frac{k}{2}$. A contradiction. This shows that $V\cap H_c$ is rotund.
\end{proof}

\begin{Rem} \label{rem-generic-hyper}
It follows from the above lemma, by induction, that if $V \subset (K^*)^n$ is a rotund variety defined over $C \subset K$ of dimension $d$ and $H_{n,k}(x,c^1,\dots,c^k)$ is the intersection of $k$ generic hyperplanes over $C$ with $d-k \geq \frac{n}{2}$, then $V \cap H_{n,k}(x,c^1,\dots,c^k)$ is a rotund variety of dimension $d-k$. 

In particular, the subvariety $H_{n,k}(x,c^1,\dots,c^k)$ of $(K^*)^n$ defined by the intersection of $k$ generic hyperplanes with $k \leq \frac{n}{2}$ is rotund.
\end{Rem}

\noindent\textit{Third example: Generic hyperplanes in the elliptic curve case.} 
Let us now consider analogues of the rotund varieties in the previous example in the case where $\Alg$ is an elliptic curve.

As mentioned before, we consider $A = \Alg(K)$ as a subvariety of $\PP^2(K)$ whose affine part is defined by an equation $y^2 = 4 x^3 +\alpha x + \beta$, with $\alpha, \beta \in k_0$.

The above arguments about generic hyperplanes in the multiplicative group case can be easily adapted to show the following: Let $V$ be a rotund subvariety of $A^n$ defined over $C$ of dimension $d$ with $d-1 \geq \frac{n}{2}$. If $H$ is a hyperplane in $K^{2n}$, generic over $C$, then the intersection of $V$ and the Zariski closure of $H$ in $(\PP^2)^n$ is a rotund subvariety of $A^n$ of dimension $d-1$. 

Also, if $H$ is the intersection of $k$ hyperplanes in $K^{2n}$, generic over $C$, with $d-k \geq \frac{n}{2}$, then the intersection of $V$ and the Zariski closure of $H$ in $(\PP^2)^n$ is a rotund subvariety of $A^n$ of dimension $d-k$. In particular, the intersection of $A^n$ and the Zariski closure in $(\PP^2)^n$ of the intersection of $k$ generic hyperplanes in $K^{2n}$, where $k \leq \frac{n}{2}$, is a rotund subvariety of $A^n$ of dimension $n-k$.\\

\noindent\textbf{Definability of rotundity.}
\begin{Lemm} \label{lemm-rotundity}
For every subvariety $W(x,y)$ of $\Alg^{n+k}$ defined over $k_0$, there exists a quantifier-free $L_{\Alg}$-formula $\theta(y)$ such that for all $A\models T_{\Alg}$ and all $c \in A^k$,
\[
A \vDash \theta(c) \iff W(x,c)\text{ is rotund}.
\]
\end{Lemm}

\begin{proof}
Let $W(x,y)$ be a subvariety of $\Alg^{n+k}$ defined over $k_0$. Let $C_1,\dots,C_s$ be proper algebraic subgroups of $A^n$ as provided by the Weak CIT (\ref{weakcit}) for the family of subvarieties of $A^n$ defined by $W(x,y)$. For each $i= 1.\dots,s$, let $M^i$ be an $n_i \times n$-matrix with entries in $\End(\Alg)$ of rank $n_i$ such that $C_i$ is defined by the system of equations $M^i \cdot x = 0$.

Let $\theta(y)$ be the conjunction of the following:
\begin{itemize}
\item a quantifier-free $L_{\Alg}$-formula $\theta_0(y)$ such that $\theta_0(c)$ holds if and only if the variety $W_c$ is irreducible and has dimension $\geq \frac{n}{2}$,
\item for each $i=1,\dots,s$, a quantifier-free formula $\theta_i(y)$ such that $\theta_i(c)$ holds if and only if the dimension of $M^i \cdot W_c$ is at least $\frac{n_i}{2}$.
\end{itemize}

The existence of the formulas $\theta_i(y)$, $i=0,\dots,s$ is given by the following facts: that the theory of algebraically closed fields of any given characteristic (in this case 0) has the \emph{definable multiplicity property} (Lemma 3 in \cite{HruFusion}), which transfers to the theory $T_{\Alg}$, since the bi-interpretation is rank preserving, to give that the irreducibility of $W_c$ is a definable property on $c$; the definability of Morley rank in strongly minimal theories (Corollary~5.6 in \cite{ZieBous}), which here corresponds to the definability of dimension. We also use that the theory $T_{\Alg}$ has quantifier elimination. 

It is clear that for all $c$, if $W(x,c)$ is rotund, then $\theta(c)$ holds. 

To prove the converse, suppose towards a contradiction that we have $c$ such that $\theta(c)$ holds but $W(x,c)$ is not rotund. We can then find a $k\times n$-matrix $M$ with entries in $\End(\Alg)$ of rank $k \geq 1$ such that $\dim M \cdot W(x,c) < \frac{k}{2}$. Let $C$ be the algebraic subgroup of $A^n$ defined by the equation $M \cdot x = 0$.

Let $b$ be a generic point of $W_c$ over $k_0(c)\alg$ and let $S$ be an irreducible component of $W_c \cap b + C$ containing $b$. 

Note that $\dim S = \dim W_c \cap b + C$. Indeed, we have
\[
\dim S \geq \trd(b/c (M \cdot b)) = \trd(b/c) - \trd(M \cdot b/c) = \dim W_c - \dim M \cdot W_c,
\]
and, by the theorem on the dimension of fibres (\cite[I.6.3]{Shaf1}), $\dim W_c - \dim M \cdot W_c = \dim W_c \cap (b + C)$; hence $\dim S = \dim W_c \cap b + C$.

Since $\dim M \cdot W_c < \frac{k}{2}$; in particular, $\dim M \cdot W_c < k$. Hence
\[
\dim S = \dim W_c - \dim M \cdot W_c > \dim W_c - k.
\]
Therefore
\[
\dim S > \dim W_c - k.
\]
But
\[
\dim W_c - k = \dim W_c + (n-k) - n = \dim W_c + \dim (b + C) - n.
\]
Thus, $\dim S > \dim W_c + \dim C -n$, i.e. $S$ is an atypical component of the intersection of $W_c$ and $b + C$. 


Thus, by \ref{weakcit} (Weak CIT), there exists $i \in \{1,\dots,s\}$ and $b' \in A^n$ such that $S$ is contained in $b' + C_i$ and $S$ is a typical component of the intersection of $W_c$ and $b + C$ with respect to $b' + C_i$, i.e. $\dim S = \dim W_c \cap (b' + C_i) + \dim (b + C) \cap (b' + C_i) - \dim (b' + C_i)$.

Note that our assumption that $\dim M \cdot W_c < \frac{k}{2}$ can be written as $\dim W_c \cap (b + C) > \frac{1}{2} \dim C$. Also, if $C'$ is any algebraic subgroup of $A^n$ with $C' \subset C$ and $b + C' \supset S$, then
\[
\dim W_c \cap (b + C') \geq \dim S = \dim W_c \cap (b + C) > \frac{1}{2} \dim C \geq \frac{1}{2} \dim C'.
\]
Hence $\dim W_c \cap (b + C') > \frac{1}{2} \dim C'$. This implies that we may assume $C$ to be the minimal algebraic subgroup having a coset that contains $S$. Thus, $C \cap C_i = C$ and therefore the typicality equation becomes
\[
\dim S = \dim W_c \cap (b' + C_i) + \dim C - \dim C_i.
\]

One can then easily see the following
\begin{align*}
\dim W_c\cap (b'+C_i) &= \dim S - \dim C + \dim C_i\\
&> -\frac{1}{2} \dim C + \dim C_i\\
&\geq \frac{1}{2} \dim C_i.
\end{align*}
This implies that $\dim M^i \cdot W_c < \frac{n_i}{2}$, which contradicts the fact that $\theta(c)$ holds.
\end{proof}

\subsubsection{The EC-property}

\begin{Def}
A structure $(A,G)$ in $\CC_0$ is said to have the \emph{EC-property} if for every even $n\geq 1$ and every rotund subvariety $W$ of $A^n$ of dimension $\frac{n}{2}$, the intersection $W \cap G^n$ is Zariski dense in $W$; i.e. for every proper subvariety $W'$ of $W$ the intersection $(W \setminus W') \cap G^n$ is non-empty.
\end{Def}

\begin{Lemm}
There exists a set of $\forall\exists$-$L$-sentences $T^1$ such that for any structure $(A,G)$ in $\CC_0$
\[
(A,G) \vDash T^1 \iff (A,G)\text{ has the EC-property.}
\]
\end{Lemm}
\begin{proof}
For each even integer $n \geq 1$, and each subvariety $W(x,y)$ of $\Alg^{n+k}$ defined over $k_0$, let $\theta_W(y)$ be a formula as provided by Lemma~\ref{lemm-rotundity}. Let $T^1$ be the theory containing, for each pair of subvarieties $W(x,y)$ and $W'(x,y)$ of $\Alg^{n+k}$, the following sentence:
\begin{multline*}
\forall y \Big(
\big(\theta_W(y) \land \dim W_y = \frac{n}{2} \land W'_y \subsetneq W_y \big)\\ \rightarrow \exists x \big(W(x,y) \land \neg W'(x,y) \land \bigwedge_{i=1}^n G(x_i) \big) 
\Big)
\end{multline*}
It is clear that $T^1$ expresses the EC-property.
\end{proof}

Henceforth, let $T^1$ denote the theory defined in the above proof. Also, let $T := T^0 \cup T^1$. 

The rest of this section shows that $T$ axiomatizes the complete theory common to all rich structures in $\CC_0$.



\begin{Def}
Let $\A$ be a structure in $\CC_0$. 

$\A$ is said to be \emph{existentially closed (with respect to
strong extensions)} if for any quantifier-free formula
$\phi(x)$, if there exist a strong extension $\A'$ of $\A$
and a tuple $b'\subset A'$ such that $\A' \vDash \phi(b')$, then there
exists $b\subset A$ such that $\A \vDash \phi(b)$.

$\A$ is said to be \emph{strongly existentially closed (with respect to
strong extensions)} if for any partial type $\Phi(x)$ over a finite subset of $\A$ consisting of quantifier-free formulas, if there exist a strong extension $\A'$ of $\A$ and a tuple $b'\subset A'$ such that $\A' \vDash \Phi(b')$, then there exists $b\subset A$ such that $\A \vDash \Phi(b)$.
\end{Def}

Clearly, if $\A$ is strongly existentially closed, then $\A$ is existentially closed. Also, if $\A$ is $\omega$-saturated, then $\A$ is strongly existentially closed if and only if it is existentially closed.

\begin{Lemm} \label{lemm-ec-1}
If $\A \in \CC_0$ is existentially closed, then $\A$ has the EC-property.
\end{Lemm}
\begin{proof}
Let $\A =(A,G)\in \CC_0$ be existentially closed, with $A = \Alg(K)$. Let $W$ be a rotund subvariety of of $A^n$ of dimension $\frac{n}{2}$ and let $W'$ be a proper subvariety of $W$, both defined over $K$. We want to see that $(W \setminus W') \cap G^n$ is non-empty. Let $b$ be a generic point of $W$ over $K$ in a model $\bar A$ of $T_{\Alg}$ extending $A$. Let $Y$ be the substructure of $A'$ with domain $A + \spank(b)$. Let $(b^i)_{i \geq 1}$ be a sequence of tuples in $(\spank b)^n$ such that $b^1 = b$ and for all $i,j \geq 1$, $j b^{ij} = b^i$, and let $G(Y) : = G(A) + \langle b^i : i \geq 1 \rangle$. 
The rotundity of $W$ gives that $\Y = (Y, G(Y))$ is a strong extension of $\A$ in $\Sub\CC_0$ (as noted in \ref{isrotund}). 
By the extension property, we can find $\A'$ in $\CC_0$ with $\Y \str \A'$. Thus, $\A \str \A'$ and $b$ is a solution in $\A'$ of the quantifier-free formula $W(x) \land \neg W'(x) \land \bigwedge_{i=1}^n G(x_i)$. Since $\A$ is existentially closed, there exists a solution of the same formula in $\A$, hence $(W \setminus W') \cap G^n$ is non-empty.
\end{proof}

\begin{Lemm} \label{lemm-ec-2}
If $\A \in \CC_0$ is rich, then $\A$ is strongly existentially closed.
\end{Lemm}
\begin{proof}
Let $\A \in \CC_0$ be rich. Let $\Phi(x)$ be a quantifier-free partial type over a finite subset $c$ of $A$, $\A'$ be a strong extension of $\A$ and $b$ be a solution of $\Phi(x)$ in $\A'$. By replacing $c$ by an appropriate basis of its strong closure, we may assume that $\spank c$ is strong in $\A$. Let $\X$ be the substructure of $\A$ with domain $\spank c$. Let $\Y$ be the substructure of $\A'$ with domain $\spank bc$. Since $\X \str \A \str \A'$, by transitivity $\X \str \A'$. In particular, $\X \str \Y$. Thus, by the richness of $\A$, we can find an embedding $j$ of $\Y$ into $\A$ over $X$. Then $j(b)$ is clearly a solution of $\Phi(x)$ in $\A$.
\end{proof}

\subsubsection{Axiomatizing richness up to $\omega$-saturation}

It is clear from the definitions that the models of $T$ are precisely the structures in $\CC_0$ satisfying the EC-property.
We shall now see that the $\omega$-saturated models of $T$ are precisely the rich structures in $\CC_0$. It follows that $T$ axiomatizes the complete theory common to all rich structures in $\CC_0$.\\

\noindent\textbf{Rich structures are models of $T$.}
\begin{Rem} \label{rich-modT}
By \ref{lemm-ec-1}, every existentially closed structure in $\CC_0$ is a model of $T$. In particular, by \ref{lemm-ec-2}, every rich structure in $\CC_0$ is a model of $T$.
\end{Rem}

\noindent\textbf{$\omega$-saturated models of $T$ are rich.}
\begin{Def}
Let $\X,\Y$ be structures in $\Sub \CC_0$. Assume $\X \str \Y$. We say that $\X \str \Y$ is a \emph{minimal} strong extension if  $X \neq Y$ and there exists no $\ZZ \in \Sub\CC_0$ with $X \subsetneq Z \subsetneq Y$ such that $\X \str \ZZ \str \Y$.
\end{Def}

\begin{Rem}
Note that every strong extension $\X \str \Y$ with $\X,\Y$ in $\Fin\CC_0$ decomposes into a finite tower of minimal strong extensions, i.e. there exist a positive integer $n$ and $\X_0, \dots, \X_n \in \Fin\CC_0$ such that $\X_0 = \X$, $\X_n = \Y$ and for all $i=0,\dots,n-1$, $\X_i \str \X_{i+1}$ is a minimal strong extension.

Consequently, for all $\A\in \CC_0$, the structure $\A$ is rich if and only if for all $\X,\Y \in \Fin \CC_0$ such that $\X \str \A$ and $\X \str \Y$ is a \emph{minimal} strong extension, there exists a strong embedding of $\Y$ into $\A$ over $\X$.
\end{Rem}

\begin{Rem} \label{rem-minimal}
Let $\X,\Y \in \Fin\CC_0$ with $\X \str \Y$ and $\delta(X) = \delta(Y)$. It is easy to see that the extension $\X \str \Y$ is minimal if and only if for every $\cl_0$-closed $Z$ with $X \subsetneq Z \subsetneq Y$, $\delta(Z/X) > 0$.
\end{Rem}


\begin{Lemm} \label{min-exts}
Let $\X,\Y$ be structures in $\Fin\CC_0$ such that $\X \str \Y$ is a minimal strong extension. Then exactly one of the following holds:
\begin{enumerate}

\item \emph{(Prealgebraic extension)} There exists an even number $n \geq 2$ and a $\cl_0$-basis $b \in G(Y)^n$ of $Y$ over $X$ with $\trd(b/X) = \frac{n}{2}$, and hence $\delta(b/X) = 0$.

\item \emph{(Green generic extension)} There is $b_0\in G(Y) \setminus X$ such that $Y = X + \spank b_0$ and $b_0$ is transcendental over $X$, hence $\delta(b_0/X) = 1$.

\item \emph{(Algebraic extension)} $G(Y) = G(X)$ and there is $b_0 \in Y \setminus X$ such that $Y = X + \spank b_0$ and $b_0$ is algebraic over $X$, hence $\delta(b_0/X) = 0$.

\item \emph{(White generic extension)} $G(Y) = G(X)$ and there is $b_0 \in Y \setminus X$ such that $Y = X + \spank b_0$ and $b_0$ is transcendental over $X$, hence $\delta(b_0/X) = 2$.
\end{enumerate}
\end{Lemm}
\begin{proof}
Let $\X,\Y$ be structures in $\Fin\CC_0$ such that $\X \str \Y$ is a minimal strong extension. Then exactly one of the following cases occurs:

\noindent\textbf{Case 1:} $G(Y) \neq G(X)$ and $\delta(Y/X)= 0$.

First note that, by minimality, $Y = X + \spank G(Y)$. Let $n$ be the
linear dimension of $Y$ over $X$, which is also the linear dimension of $G(Y)$ over $G(X)$, by modularity. Then, since $\delta(Y/X) = 0$, we have $\trd(Y/X) = \frac{n}{2}$. 
Thus, if $b \in G(Y)^n$ is a $\cl_0$-basis of $Y$ over $X$, then $b$ is also a $\cl_0$-basis of $G(Y)$ over $G(X)$ and $\trd(b/X) = \frac{n}{2}$.

\noindent\textbf{Case 2:} $G(Y) \neq G(X)$ and $\delta(Y/X)> 0$.

As in the previous case, $Y = X + \spank G(Y)$. Moreover, 
\[
\ld(Y/X) = \ld(G(Y)/G(X)) = 1.
\]
To see this, take an element $b_0$ in $G(Y) \setminus G(X)$, and note that
$\trd(b_0/X) = 1$, and hence $\delta(b_0/X) = 1$. Now, if $\ld(Y/X)$ is strictly greater than one then we get a tower of proper strong extensions: either $X \str X+\spank b_0 \str Y$ is such
a tower, or otherwise there exists a tuple $b \subset Y$ starting with the element $b_0$ such
that $\spank b_0 \subsetneq \spank b \subsetneq Y$ and $\delta(\spank b/
\spank b_0) < 0$, and hence $\delta(X+\spank b) = \delta(X)$, from which it
follows that $X \str X+\spank b \str Y$ is a tower of (proper) strong extensions.

Thus, for any element $b_0\in G(Y) \setminus X$, $Y = X + \spank b_0$.

\noindent\textbf{Case 3:} $G(Y) = G(X)$ and $\delta(Y/X) = 0$.

By the minimality of the extension, $Y = X + \spank b_0$ for any $b_0 \in Y \setminus X$. Also, since $\delta(Y/X) = 0$, any such $b_0$ is algebraic over $X$.

\noindent\textbf{Case 4:} $G(Y) = G(X)$ and $\delta(Y/X) > 0$.

As in the previous case, $Y = X + \spank b_0$, for any $b_0 \in Y \setminus X$. Thus, $\trd(Y/X) \leq 1$ and, since we are assuming $\delta(Y/X)> 0$, such $b_0$ must be transcendental over $X$, and hence $\delta(Y/X) = 2$. 
\end{proof}

We shall henceforth use the names given in the above lemma to the different kinds of extensions. For the sake of brevity let us also stress the language by making the following definition:

\begin{Def}
Let us say that a rotund subvariety $W$ of $A^n$ is \emph{prealgebraic minimal} if it has dimension $\frac{n}{2}$ and for every $1 \leq k < n$, and every $k \times n$-matrix with entries in $\End(\Alg)$ of rank $k$, $\dim M \cdot W > \frac{k}{2}$.
\end{Def}

It is easy to see that a rotund variety $W$ is prealgebraic minimal if and only if any strong extension constructed from $W$ as in Remark~\ref{isrotund} is a prealgebraic minimal extension.


%
Using Remark~\ref{rem-minimal}, a minor modification of the proof of Lemma~\ref{lemm-rotundity} (strengthening the formula $\theta_0(y)$ to require $\dim W_y = \frac{n}{2}$ and each $\theta_i(y)$ to correspond to the \emph{strict} inequality $\dim M^i \cdot W_c > \frac{n_i}{2}$) yields the following
:
\begin{Lemm} \label{lemm-rotundity-premin}
For every subvariety $W(x,y)$ of $\Alg^{n+k}$ defined over $k_0$, there exists a quantifier-free $L_{\Alg}$-formula $\theta'(y)$ such that for all $A\models T_{\Alg}$ and all $c \in A^k$,
\[
A \vDash \theta'(c) \iff W(x,c) \text{ is a prealgebraic minimal rotund subvariety}.
\]
\end{Lemm}

%
The following two remarks use the above definability lemma and model-theoretic arguments to find new rotund varieties from the ones coming from intersections of generic hyperplanes. These results will then be applied in the proof of  Lemma~\ref{satmodT-rich}.

\begin{Rem} \label{hyper-over-alg}
Let us first consider the case where $\Alg$ is the multiplicative group. For each $n \geq 1$, if $H_{2n,n}(x,c^1,\dots,c^n)$ is the intersection of $n$ generic hyperplanes in $A^{2n}$ (i.e. $\trd(c^1,\dots,c^n) = 2n^2$), then $H_{2n,n}(x,c^1,\dots,c^n)$ is a prealgebraic minimal rotund variety.

Indeed, we already know, by \ref{rem-generic-hyper}, that $H_{2n,n}(x,c^1,\dots,c^n)$ is a rotund variety of dimension $n$. Now consider a generic point $b$ of $H_{2n,n}(x,c^1,\dots,c^n)$ over $C:= \bigcup c^i$ and let $b'$ be a subtuple of $b$ of length $k$ with $1 \leq k < 2n$. Using the algebraic independence of $C$ one sees the following: if $1\leq k \leq n$, then $\trd(b'/C) = \trd(b/C) - \trd(b/C b') \geq n - (n-k) = k > \frac{k}{2}$; if $n < k < 2n$, then $\trd(b'/C) = \trd(b/C) - \trd(b/C b') \geq n - 0 = n > \frac{k}{2}$. This shows that $H_{2n,n}(x,c^1,\dots,c^n)$ is prealgebraic minimal.

Then, by Lemma~\ref{lemm-rotundity-premin} and the model-completeness of $T_{\Alg}$, it follows that for each $n \geq 1$, there exist tuples $c^{1*},\dots,c^{n*} \in ((\Q\alg)^*)^n$ such that $H_{2n,n}(x,c^{1*},\dots,c^{n*})$ is a prealgebraic minimal rotund subvariety of $A^{2n}$.

As noted earlier, in the elliptic curve case we have an analogue of each generic hyperplane $H_{2n,n}(x,c^1\dots,c^n)$ with algebraically independent $c^1\dots,c^n$ over $k_0$, namely the Zariski closure of the intersection of the generic hyperplane $H_{4n,2n}(x,c'^1\dots,c'^{2n})$, where $c'^1,\dots, c'^{2n}$ are algebraically independent over $k_0$, and $A^n$. Therefore we also have an analogue of the varieties $H_{2n,n}(x,c^{1*},\dots,c^{n*})$ found above, i.e. prealgebraic minimal rotund subvarieties of $A^{2n}$ defined over $k_0\alg$. To ease the notation, we shall henceforth, also in the case where $\Alg$ is an elliptic curve, use $H_{2n,n}(x,c^{1*},\dots,c^{n*})$ to denote such a subvariety of $A^{2n}$.
\end{Rem}

\begin{Rem} \label{g-alg}
Let $\A = (A,G)$ be a model of $T$. Let us see that $A = \acl_{T_{\Alg}} (G)$.  

If $\A$ is the multiplicative group, then a stronger statement is easy to prove: by the rotundity of the varieties defined by the equations $X+Y=c$ with $c \in A = K^*$, that $\A$ satisfies the EC-property directly implies that $G+G = A$. 

Assume now that $\Alg$ is an elliptic curve. 

Let $c \in K$, $c \neq 0$. Let $q_1,\dots,q_4$ be algebraically independent over $k_0(c)$. The set $H(x,y,q,c)$ defined by the equation $q_1x_1 + q_2y_1 + q_3x_1 + q_4y_2 = c$ is a generic hyperplane in $K^4$. Therefore this equation defines a rotund subvariety of $A^2$. By the algebraic independence assumption, the algebraic type of $q$ over $k_0(c)\alg$ does not fork over $k_0\alg$. An alternative way of expressing this last fact is that the algebraic type of $q$ over $k_0(c)\alg$ is a coheir of its restriction to $k_0\alg$, which means that it is finitely satisfiable in $k_0\alg$. Combining this finite satisfiability with the fact that the rotundity of $H(x,y,q,c)$ is definable on the parameters $q,c$ (in the algebraic language), we see that there exists $q^*_1,\dots,q^*_4 \in k_0\alg$ such that the equation $q^*_1x_1 + q^*_2y_1 + q^*_3x_1 + q^*_4y_2 = c$ defines a rotund subvariety of $A^2$. Since $\A$ satisfies the EC-property, this subvariety has a solution in $G^2$. This implies that $c$ is in $\acl_{\Alg}^{\text{eq}}(G)$. It follows that $K \subset \acl_{\Alg}^{\text{eq}}(G)$ and, therefore, $A = \acl_{\Alg}(G)$.

Additionally, let us note the following: suppose we start the above argument with $c \in K$ transcendental over $k_0$ and thus find $q^*_1,\dots,q^*_4 \in k_0\alg$ such that $q^*_1x_1 + q^*_2y_1 + q^*_3x_1 + q^*_4y_2 = c$ defines a rotund subvariety of $A^2$. Then, for any $c' \in K$ transcendental over $k_0$, since $c$ and $c'$ are conjugates over $k_0\alg$, the subvariety of $A^2$ defined by $q^*_1x_1 + q^*_2y_1 + q^*_3x_1 + q^*_4y_2 = c'$, with \emph{the same} $q^*_1,\dots,q^*_4$ as before, is also rotund. This is used in the proof of the following lemma.
\end{Rem}


\begin{Lemm} \label{satmodT-rich}
If $\A$ is an $\omega$-saturated model of $T$, then $\A$ is a rich structure in $\CC_0$.
\end{Lemm}
\begin{proof}
Assume $\A = (A,G)$ is an $\omega$-saturated model of $T$. Since $T$ contains $T^0$, which by definition axiomatizes the class $\CC_0$, it is clear that $\A$ is in $\CC_0$.

To show that $\A$ is rich, let $\X,\Y$ be structures in $\Fin\CC_0$ such that $\X$ is strong in $\A$ and $\X \str \Y$ is a minimal strong extension. We need to show that there is a strong embedding of $\Y$ into $\A$ over $X$. We consider different cases for the different kinds of minimal strong extensions, as determined in Lemma~\ref{min-exts}.

\noindent\textbf{Case 1:} $\X \str \Y$ is a prealgebraic minimal extension.

Let $b \in G(Y)^n$ be a $\cl_0$-basis of $Y$ over $X$. 

Let $(b^i)_{i \geq 1}$ be a sequence of tuples in $G(Y)^n$ such that $b^1 = b$ and for all $i,j \geq 1$, $j b^{ij} = b^i$. Note that $G(Y) = G(X) + \langle b^i : i \geq 1 \rangle$. It is sufficient to show that $\qftp((b^i)_i/X)$ is realised in $\A$ to obtain an embedding of $Y$ into $\A$ over $X$. Moreover, since in this case $\delta(X) = \delta(Y)$, any such embedding is necessarily strong.

Let us show that indeed $\qftp((b^i)_i/X)$ is realised in $\A$. By the $\omega$-saturation of $\A$, it is sufficient to show that for each $n \geq 1$, the type $\qftp((b^i)_{i\leq n}/X)$ is realised in $\A$: Fix $n \geq 1$. Let $N = \prod_{i \leq n} i$. Let $V := \locus(b^N / k_0(X)\alg)$. By Remark~\ref{isrotund}, $V$ is a rotund subvariety of $A^n$. Notice that, since $\delta(b^N/X) = \delta(Y/X) =0$, we have $\dim V = \frac{n}{2}$. Since $\A$ is a model of $T$, for every proper subvariety $V'$ of $V$ over $k_0(X)\alg$, the intersection $(V \setminus V') \cap G^n$ is non-empty. By the $\omega$-saturation of $\A$, it follows that we can find a generic point $b'^N$ of $V$ over $k_0(X)\alg$ in $G^n$. Note that $b'^N$ is a realisation of $\qftp(b^N/X)$. For each $i \leq n$, let $(b'^i) := \frac{N}{i} b'^N$. Then $(b'^i)_{i}$ is a realisation of $\qftp((b^i)_{i\leq n}/X)$ in $\A$.

\noindent\textbf{Case 2:} $\X \str \Y$ is a green generic extension.

This amounts to finding an element $b_0 \in \A$ such that $\dd(b_0/X) = 1$, or, equivalently, $b_0 \not\in \cl_{\dd}(X)$.

By \ref{g-alg}, we can find a finite dimensional $X' \subset G$ such that $X \subset \acl_{T_{\Alg}}(X')$. Since always $\acl_{T_{\Alg}} \subset \cl_{\dd}$, we also have $\cl_{\dd}(X) \subset \cl_{\dd}(X')$. It therefore suffices to find $b_0 \in \A$ outside $\cl_{\dd}(X')$. Thus, we may assume that $X$ is contained in $G$.

There is a partial type $\Phi(x_0)$ over $X$ expressing the following: $x_0 \in G$; $x_0 \not\in \acl_{T_{\Alg}}(X)$; for each $n \geq 1$, for all $x_1,\dots,x_n \in G$, $\delta(x_0,\dots,x_n/X) > 0$. 

Indeed, for each $n \geq 1$, the condition above can be expressed by a set of formulas in the variables $x_1,\dots,x_n$. To see this one can slightly modify the argument in the proof of \ref{lemm-ss-def}, this time looking at varieties of dimension $\leq \frac{n}{2}$, instead of $<\frac{n}{2}$, which one can see makes no essential difference. Note that here we use our assumption that $X$ is contained in $G$, for it is only in this case that we know how to express the predimension inequalities.

It is now sufficient to find a realisation for the type $\Phi(x_0)$ in $\A$. By the $\omega$-saturation of $\A$, it suffices to show that $\Phi(x_0)$ is finitely satisfiable in $\A$. For this, it is in turn sufficient to show that for each $n \geq 1$, there is a prealgebraic \emph{minimal} extension $\X_n$ of $\X = (X,G\cap X)$ with $\ld(X_n/X) >n$ (and apply the result of Case 1).

Such extensions can be found as follows: For each $n \geq 1$, let $H_{2n,n}(x, c^{*1},\dots,c^{*n})$ be as in (the end of) Remark~\ref{hyper-over-alg} and let $b$ be a generic of $H_{2n,n}(x, c^{*1},\dots,c^{*n})$ over $k_0(X)\alg$. Put $X_n := X + \spank b$ and $G_n := G(X) + \langle b^i : i \geq 1 \rangle$, where $(b^i)_{i \geq 1}$ is a sequence of tuples such that $b^1 = b$ and for all $i,j \geq 1$, $j b^{ij} = b^i$. Then $\X_n = (X_n,G_n)$ is a prealgebraic minimal extension of $\X$ and $\ld(X_n/X) = 2n > n$.

\noindent\textbf{Case 3:} $\X \str \Y$ is an algebraic extension.

Let $b_0$ be any element of $Y \setminus X$. Let $b_0'\in A$ be a root of the minimal polynomial of $b_0$ over $X$. Then $\Y$ is isomorphic to the substructure of $\A$ with domain $Y':= X + \spank b_0'$. Since $\delta(Y') = \delta(Y) = \delta(X)$, $Y'$ is strong in $\A$.

\noindent\textbf{Case 4:} $\X \str \Y$ is a white generic extension.

It suffices to find an element $b_0$ in $\A$ with $\dd(b_0/X) = 2$. By the argument in Case 2, we can find $b_1 \in A$ with $\dd(b_1/X) = 1$ and $b_2 \in A$ with $\dd(b_2/X b_1) = 1$. Note that, by additivity, $\dd(b_1,b_2/ X) = 2$, 

We shall now find such a $b_0$. 

Let us first deal with the case where $\Alg$ is the multiplicative group. Let $b_0 := b_1 + b_2$ and let us show that indeed $\dd(b_0/X) = 2$. Let us consider two cases:

\noindent\emph{Case 1: $b_1 \in \acl_{T_\Alg}(\scl(X b_0))$.} Then $\dd(b_1/X b_0) = 0$. And also $b_2 \in \acl_{T_\Alg}(\scl(X b_0))$, hence $\dd(b_2/X b_0) = 0$. It follows that $\dd(b_0/X) = \dd(b_1,b_2/X) = 2$.

\noindent\emph{Case 2: $b_1 \not\in \acl_{T_\Alg}(\scl(X b_0))$.} Then $(b_1,b_2)$ is a generic point of the variety defined by $X+Y = b_0$ over $\scl(X b_0)$. Therefore $\delta(b_1,b_2/\scl(X b_0)) = 0$, and hence $\dd(b_1,b_2/X b_0) = 0$. Thus, $\dd(b_0/X) = 2$.

Let us now show how to do an analogous argument in the elliptic curve case. Since $b_1$ and $b_2$ are independent in the sense of $\cl_{\dd}$, they are also independent in the sense of $\acl_{\Alg}$. In particular, $b_1$ and $b_2$ are not algebraic over the empty set in the sense of $T_{\Alg}$ and therefore lie in the affine part of the elliptic curve $A$ (the unique point at infinity is algebraic). Let us thus write $b_1 = (b_{11}, b_{12})$ and $b_2 = (b_{21}, b_{22})$, with each $b_{ij}$ in $K$. 

By Remark~\ref{g-alg}, we can find $q^*_1,\dots,q^*_4 \in k_0\alg$ such that for any $c\in K$ transcendental over $k_0$, the equation $q^*_1x_1 + q^*_2y_1 + q^*_3x_2 + q^*_4y_2 = c$ defines a rotund subvariety of $A^2$. Define $b_{01} := q^*_1 b_{11} + q^*_2 b_{12} + q^*_3 b_{21} + q^*_4 b_{22}$, and let $b_{02} \in K$ be such that the point $b_0 := (b_{01},b_{02})$ is in $A$. Since $b_1$ and $b_2$ are $\acl_{\Alg}$-independent, $b_{11}$ and $b_{12}$ are algebraically independent over $k_0$. Therefore $b_{01}$ is transcendental over $k_0$. Hence, the equation $q^*_1x_1 + q^*_2y_1 + q^*_3x_2 + q^*_4y_2 = b_{01}$ defines a rotund subvariety of $A^2$ having $(b_1,b_2)$ as a solution. From here one can follow the same argument as in the multiplicative group case (considering two cases, etc.), to show that $\dd(b_0/X) = \dd(b_1,b_2/X) = 2$.
\end{proof}

\noindent\textbf{Rich structures are $\omega$-saturated.}
\begin{Prop} \label{iff} 
The rich structures are precisely the $\omega$-saturated models of $T$.
\end{Prop}
\begin{proof}
By Lemma~\ref{satmodT-rich}, every $\omega$-saturated model of $T$ is rich.

Let us now show that every rich structure is $\omega$-saturated. Let
$\A$ be a rich structure. Let $\A'$ be an $\omega$-saturated
elementary extension of $\A$. By Lemma~\ref{satmodT-rich}, $\A'$ is rich
and, by Lemma~\ref{bnf}, $\A$ and $\A'$ are $L_{\infty\omega}$-equivalent. Since $L_{\infty\omega}$-equivalence preserves $\lambda$-saturation for all infinite cardinals $\lambda$, $\A$ is $\omega$-saturated.
\end{proof}

\noindent\textbf{Completeness of $T$.}

\begin{Prop} \label{complete} 
The theory $T$ is complete.
\end{Prop}
\begin{proof}
Firstly note that the existence of rich structures in $\CC_0$ was proved in Section~\ref{sec:existence}, and we therefore know that $T$ is consistent. Completeness follows immediately from Proposition~\ref{iff} and Remark~\ref{RemF}.
\end{proof}


\section{Model-theoretic properties} \label{sec:mth}

In this section we prove that the theory $T$ is $\omega$-stable and near model complete. We also calculate $U$-ranks and Morley ranks in the theory $T$. For missing definitions and basic facts from stability theory we refer to Section~1 of \cite{PlGStab} and \cite{PzMT}.

Throughout this section we work in a monster model $\bar \A = (\bar A,G)$ of $T$ (sufficiently saturated and strongly homogeneous).

\subsection{$\omega$-stability}

Here we shall use Lemma~\ref{bnf}, which is indeed a form of quantifier elimination, to show that the theory $T$ is $\omega$-stable. We follow the proof of Poizat in \cite{PzEq3}, but we note that an application of the Thumbtack Lemma (Theorem~\ref{ThumbtackLemma} below) is necessary for the argument to yield the full result).



In order to state the Thumbtack Lemma we make the following definition.

\begin{Def} \label{Kummergen}
Let $K/k$ be an extension of algebraically closed fields, both extending $k_0$. A tuple $a \subset A = \Alg(K)$ is said to be \emph{Kummer generic over $k$}, if every $\End(\Alg)$-module automorphism of $\spank(\Alg(k) + \langle a \rangle)$ fixing $\Alg(k) + \langle a \rangle$ pointwise is induced by a field automorphism of $k(a)^{\alg}$ over $k_0$.
\end{Def}

The above terminology is taken from \cite{HilsGreenGen}. In fact, the definition here differs from the one in \cite{HilsGreenGen}, Definition 4.1, but our use of the term is legitimized by Fact 4.2 there.

To explain the concept of Kummer genericity, let $a$ be a tuple of elements of $A = \Alg(K)$ and consider a sequence $(a^i : i \geq 1)$ of tuples in $A$ with $a^1 = a$ and $(a^{ij})^i = a^j$ for all $i,j \geq 1$. If $V$ is the locus of $a$ over $k$, then for every $i$, $a^i$ lies in the variety $\frac{1}{i} V$, that is the inverse image of $V$ under the map $x \mapsto ix$ (with respect to the group operation on $\A$). Since, in general, the variety $\frac{1}{i} V$ is not irreducible, the type of $a^i$ over $k$ is, in general, not determined by that of $a$.

However, it is easy to see that $a$ is Kummer generic over $k$ if and only if all sequences $(a^i : i \geq 1)$, with $a^1 = a$ and $(a^{ij})^i = a^j$ for all $i,j \geq 1$, are conjugated by a field automorphism of $k(a)^{\alg}$ over $k_0$. Thus, in other words, $a$ is Kummer generic over $k$ if and only if the type of $a$ over $k$ determines the type over $k$ of any sequence sequence $(a^i : i \geq 1)$ as above.

\begin{Thm}[Thumbtack Lemma] \label{ThumbtackLemma}
Let $\Alg$ be the multiplicative group or an elliptic curve. Let $k$ and $K$ be algebraically closed fields extending $k_0$ with $k \subset K$.

For every $a \in A = \Alg(K)$, there exists a $k_{\Alg}$-linear basis $a'$ of $\spank(\Alg(k) \cup a)$ over $\Alg(k)$ that is Kummer generic over $k$.
\end{Thm}

In the multiplicative group case, the above theorem is Theorem 2.3 in \cite{BZCovers} (which builds upon \cite{ZCovers}). For elliptic curves without complex multiplication and defined over a number field $k_0$, it is case $N=1$ of Lemma 4.2.1.iii in \cite{Bays}. The result holds for arbitrary semiiabelian varieties by a theorem of Bays, Gavrilovich and Hils (\cite{BGH}[Theorem 1.1]), which in fact follows from a version by the same authors in the more general context of groups of finite Morley rank (\cite{BGH}[Theorem 6.4]).



\begin{Lemm} \label{scl-acl}
For every subset $B$ of $\bar \A$, $\acl_T(B) = \acl_{T_\Alg}(\scl(B))$.

Moreover, every algebraically closed subset of $\bar \A$ is strong.
\end{Lemm}
\begin{proof}
For the first part, it is sufficient to show that the equality holds for finite $B$. Thus, fix a finite $B \subset \bar \A$

($\supset$) Let $a$ be a $\cl_0$-basis of $\scl(B)$. It suffices to show that $a \subset \acl_T(B)$. Every automorphism of $\bar \A$ fixing $B$ pointwise fixes $\scl(B)$ set-wise. Thus, all conjugates of $a$ over $B$ are contained in $\scl(B)$, and hence there are at most countably many of them. By the saturation of $\bar \A$, we get that in fact there must be only finitely many. This shows that every element in $a$ is algebraic over $B$.


($\subset$) Let $a$ be an element of $\bar A \setminus \acl_{\Alg}(\scl(B))$ and let us show that $a$ is not in $\acl_T(B)$.

Let $C = \scl(B)$ and $D = \scl(aB)$. For each $n \geq 1$, let $D_n$ be the free amalgam of $n$ isomorphic copies of $D$ over $C$. By the richness of $\bar \A$, each $D_n$ embeds strongly into $\bar \A$ over $C$. The different copies of $a$ in each $D_n$ have all the same type, this is because the different copies of $D$ are strong in $D_n$ and hence are strongly embedded in $\bar \A$, so the isomorphisms between them are elementary maps. 
It follows that the type of $a$ over $B$ has infinitely many realisations in $\bar \A$. Therefore $a$ is not in $\acl_T(B)$.

The second part of the statement follows immediately from the first and Lemma~\ref{lemm-ss-acl}.
\end{proof}

\begin{Lemm} \label{eq-acl}
For all sequences $a,a'\subset \bar A$, possibly infinite, the following are equivalent:
\begin{enumerate}
\item $\tp(a) = \tp(a')$,
\item The map $a \mapsto a'$ extends to an $L$-isomorphism from $\acl_T(a)$ onto $\acl_T(a')$.
\item The map $a \mapsto a'$ extends to an $L$-isomorphism from $\scl(a)$ onto $\scl(a')$.
\item The map $a \mapsto a'$ extends to an $L$-isomorphism between strong subsets of $\bar \A$,
\end{enumerate}
\end{Lemm}
\begin{proof}
That (1.) implies (2.) holds for arbitrary first-order theories, by an easy argument.

To see that (2.) implies (3.), note that any isomorphism from $\acl_T(a)$ onto $\acl_T(a')$ sending $a$ to $a'$ has to map $\scl(a)$ onto $\scl(a')$. Indeed, to see this simply notice that, since $\acl_T(a)$ is strong in $\bar \A$ (by Lemma~\ref{scl-acl}), $\scl(a)$ is the intersection of all finite dimensional $\cl_0$-closed subsets $B$ of $\acl_T(a)$ that are strong in $\acl_T(a)$, and similarly for $a'$.

It is clear that (3.) implies (4.).

The implication from (4.) to (1.) follows immediately from Lemma~\ref{bnf}.
\end{proof}

\begin{Thm} \label{Stab}
The theory $T$ is $\omega$-stable.
\end{Thm}
\begin{proof}
Let $\lambda$ be an infinite cardinal. We shall see that for all $B \subset \bar \A$ with $|B| \leq \lambda$, there are no more than $\lambda$ complete 1-types over $B$; thus showing that $T$ is $\lambda$-stable for all infinite $\lambda$, i.e. that it is $\omega$-stable.

Let $B \subset \bar A$ be of cardinality $\lambda$. By passing to the algebraic closure $\acl_T(B)$ of $B$, we may assume that $B$ is a strong subset of $\bar \A$ of the form $\Alg(K)$ for some algebraically closed subfield $K$ of $\bar K$, where $\bar K$ is an algebraically closed field such that $\bar A = \Alg(\bar K)$.

Consider an arbitrary element $a_0$ in $\bar A$. Recall that by Lemma~\ref{eq-acl}, the type of $a_0$ over $B$ is determined by the isomorphism type of $\scl(Ba_0)$ over $B$. 

Let $a$ be a $\cl_0$-basis of $\scl(Ba_0)$ over $B$. Note that $a$ is a finite tuple: working over (the \emph{strong} set) $B$, the set $\scl(Ba_0)$ is the strong closure of a finite set and is thus finite dimensional (over $B$). Suppose $(a^i)_{i\geq 1}$ is a sequence of tuples such that (1) $a^1 = a$ and for all $i,j \geq 1$, $j a^{ij} = a^i$, and such that (2) if a coordinate $a_j$ of $a$ is in $G$, then for all $i \geq 1$, $a^i_j$ is also in $G$. Then the isomorphism type of $\scl(Ba_0)$ over $B$ is determined by the quantifier-free type of the sequence $(a^i)_{i\geq 1}$ over $B$.

Thus, there are at most as many possibilities for the type of the element $a_0$ over $B$ as possible quantifier-free types over $B$ of sequences $(a^i)_{i \geq 1}$ with the above properties. Let us now, at first, see that there are at most $\lambda \cdot 2^{\aleph_0}$ possibilities for the quantifier-free type of such sequence $(a_i)_{i \geq 1}$ over $B$. Indeed, if we fix the algebraic type of $a$ over $B$ then then there are only finitely many possibilities for the algebraic type of each $a^i$ over $B$. Therefore there are at most $\lambda \cdot 2^{\aleph_0}$ possibilities for the algebraic type of the sequence $(a^i)_{i\geq 1}$ over $B$. Also, there are only finitely many possibilities for the colouring of the tuple $a$, which determines the colouring of the sequence $(a^i)_{i \geq 1}$. Thus, there are no more than $\lambda \cdot 2^{\aleph_0}$ possibilities for the quantifier-free type of such a sequence $(a^i)_{i \geq 1}$. This shows that the theory $T$ is $\lambda$-stable for all $\lambda \geq 2^{\aleph_0}$ and hence superstable.


Let may now improve on the above to see that, in fact, the theory $T$ is $\omega$-stable. Indeed, by the Thumbtack Lemma (Theorem~\ref{ThumbtackLemma}), we may assume without loss of generality that $a$ is Kummer generic over $B$. Thus, the algebraic type of $a$ over $B$ determines the algebraic type of the whole sequence $(a^i)_{i \geq 1}$. Therefore there are at most $\lambda$ possible quantifier-free types over $B$ of sequences $(a^i)_{i \geq 1}$ as above. This shows that $T$ is $\omega$-stable.
\end{proof}

\subsection{Near model completeness}

Let $L^*$ be the expansion of the language $L$ by a predicate for each existentially definable set in $L$. As with any extension of the language by definable predicates, $T$ extends canonically to a complete $L^*$-theory $T^*$. Recall that the $L$-theory $T$ is said to be \emph{near model complete} if $T^*$ has quantifier elimination.


The following proposition isolated a sufficient condition for near completeness of the theory $T$.

\begin{Prop} \label{nearmodelcomplete}
Assume the following property holds: For all $a\subset \A \models T$, there exists an existential $L$-formula $\tau_a^\delta(x)$ such that
\begin{itemize}
\item $\A \models \tau_a^\delta(a)$, and
\item for all $a' \subset \A' \models T$, if $\A' \models \tau_a^\delta(a')$ then $\delta(a') \leq \delta(a)$.    
\end{itemize}

Then:
\begin{enumerate}
\item For all $a\subset \A \models T$, there exists an existential $L$-formula $\tau_a^{\dd}(x)$ such that
\begin{itemize}
\item $\A \models \tau_a^{\dd}(a)$, and,
\item for all $a' \subset \A' \models T$, if $\A' \models \tau_a^{\dd}(a')$ then $\dd(a') \leq \dd(a)$.    
\end{itemize}
\item 
For all $n \geq 1$, for all $r \geq 0$, there is a set $\Phi_{n,r}(x)$ of existential $L$-formulas such that for all $\A \in \CC$ and all $a \in A^n$,
\[
\delta(a) \leq r \iff \A \models \bigvee \Phi_{n,r}(a).
\]
\item For all $n \geq 1$, for all $r \geq 0$, there is a set $\Psi_{n,r}(x)$ of existential $L$-formulas such that for all $\A \in \CC$ and all $a \in A^n$,
\[
\dd(a) \leq r \iff \A \models \bigvee \Psi_{n,r}(a).
\]
\item For all models $\A_1, \A_2$ of $T^*$, every finite partial $L^*$-isomorphism from $\A_1$ to $\A_2$ preserves the dimension function $\dd$.
\item For all $\omega$-saturated models $\A_1,\A_2$ of $T^*$, every finite partial $L^*$-isomorphism from $\A_1$ to $\A_2$ extends to a member of $\F(\A_1,\A_2)$.
\item $T$ is a near model complete $L$-theory.
\end{enumerate}
\end{Prop}
\begin{proof}
\begin{enumerate}
\item It is easy to see that we can take $\tau_a^{\dd}(x)$ to be $\exists y \tau_{(a,b)}^{\delta}(x,y)$ where $b$ is such that $(a,b)$ is a $\cl_0$-basis of the strong closure of $a$.
\item Given $n,r$, put
\[
\Phi_{n,r}(x) = \{\tau_a^\delta(x) : \A \models T, a \in A^n, \delta(a) \leq r\}
\]
\item Given $n,r$, put
\[
\Psi_{n,r}(x) = \{\tau_a^{\dd}(x) : \A \models T, a \in A^n, \dd(a) \leq r\}
\]
\item Immediate from the previous part.
\item Let $f:a_1 \mapsto a_2$ be a partial $L^*$-isomorphism from $\A_1$ to $\A_2$. Let $b_1$ be an enumeration of $\scl(a_1)$. Let $\Theta(x,y)$ be the quantifier-free $L$-type of $a_1 b_1$. Notice that since $f$ is an $L^*$-partial isomorphism, $\Theta(a_2,y)$ is finitely satisfiable in $\A_2$. Therefore, by the $\omega$-saturation of $\A_2$, there exists a realisation $b_2$ in $A_2$ of $\Theta(a_2,y)$. We thus have an $L$-partial isomorphism $\hat{f}: b_1 \mapsto b_2$ extending $f$.

Also,
\[
\delta(b_2) = \delta(b_1) = \dd(a_1) = \dd(a_2),
\]
where the last equality holds by the previous part of this lemma; hence $b_2$ is strong in $\A_2$. Therefore, $\hat{f}$ is in $\F(\A_1,\A_2)$.
\item It follows easily from the previous part and Lemma~\ref{bnf} that finite partial $L^*$-isomorphisms are $L$-elementary maps, and hence $L^*$-elementary maps. Thus, $T^*$ has quantifier elimination, i.e. $T$ is near model complete.
\end{enumerate}
\end{proof}

We shall now show that the sufficient condition for near model completeness found in Prop~\ref{nearmodelcomplete} is satisfied by the theory $T$. Here we follow Lemma 10.3 in \cite{BadField}. 

\begin{Lemm} \label{semidef-generic}
Let $V(x,y)$ be a subvariety of $\Alg^{n+k}$ defined over $k_0$. There exists a quantifier-free $L_{\Alg}$-formula $\phi(x,y)$ such that for every $b \in \bar A^k$ such that $V_b$ is a prealgebraic minimal rotund subvariety of $\bar A^n$ the following holds:
\begin{itemize}
\item every generic point $a$ of $V_b$ over $b$ satisfies $\phi(x,b)$, and
\item for every $a \models \phi(x,b)$ and every strong $\cl_0$-closed set $B$ containing $b$, either $a \subset B$ or $a$ is a generic point of $V_b$ over $B$.
\end{itemize}
\end{Lemm}
\begin{proof}
This is Lemma 4.4 in \cite{BadField}. Since the lemma deals only with the definability of relative predimension, one sees that it holds without changes in our greater generality. 
\end{proof}


\begin{Lemm} \label{defbledelta}
For all $a, b \subset \bar \A$ with $\spank(b) \str \bar \A$, there exists an existential $L$-formula $\tau_{a,b}^\delta(x,y)$ such that
\begin{itemize}
\item $\A \models \tau_{a,b}^\delta(a,b)$, and
\item for all $a',b' \subset \bar \A$, if $\A \models \tau_{a,b}^\delta(a',b')$ then $\delta(a'/b') \leq \delta(a/b)$.    
\end{itemize}
\end{Lemm}
\begin{proof}
Let $B:=\spank(b)$ and $A := \spank(ab)$. Then $B \str \A$ is a finite dimensional strong extension. 

We may assume that the extension $B \str \A$ is minimal. Indeed, in the general case, the extension can be decomposed into a tower of minimal strong extensions and the conjunction of the formulas obtained for each of the minimal extensions in the tower is equivalent to an existential formula with the required properties.

We deal separately with the different cases from Lemma~\ref{min-exts}. Cases 2,3,4 are easy: if the extension $B \str \A$ is a green generic extension (case 2), then we can take the formula $\tau_{a,b}^\delta(x,y)$ to be $G(x)$; if the extension is algebraic (case 3) then we can take $\tau_{a,b}^\delta(x,y)$ to be any formula witnessing the algebraicity of $a$ over $b$; if the extension is white generic (case 4), then $\tau_{a,b}^\delta(x,y)$ can be the formula $x=x$.

Let us thus assume that $B \str \A$ is a prealgebraic minimal extension (case 1). Let $n$ be the linear dimension of $A$ over $B$. Let $c \subset G^n$ be a green linear basis of $A$ over $B$. Let $V(z,d)$ be the locus of $c$ over $\Q(b)\alg$ and let $\tau(w,b)$ be a formula isolating the type of $d$ over $b$. 
Consider the conjunction $\tilde\tau(x,y,z)$ of the following formulas:
\begin{itemize}
\item $\psi(x,z)$\ \ , where $\psi(x,z)$ is such that $\A \models \psi(a,c)$ and for all $a',c'$, if $\A \models \psi(a',c')$ then $\spank(a') = \spank(c')$;
\item $\bigwedge_{i} G(z_i)$\ \ ;
\item $\exists w (\tau(w,y) \land \phi(z,w) \land \theta'(w)$\ \ , where $\theta'(w)$ expresses that $V_w$ is prealgebraic minimal rotund (see Lemma~\ref{lemm-rotundity-premin}) and $\phi(z,w)$ is as provided by Lemma~\ref{semidef-generic} for $V(z,w)$.
\end{itemize}

Note that $\tilde\tau(x,y,z)$ is equivalent to an existential formula. Thus, there is also an existential formula equivalent to $\exists z \tilde\tau(x,y,z)$; which works as $\tau_{a,b}^\delta(x,y)$ as we shall now see. Indeed, one can simply note that $\A \models \tilde\tau(a,b,c)$, and, if $\A \models \tilde\tau(a',b',c')$ then 
\[
\delta(a'/b') = \delta(c'/b') \leq 0,
\] 
(the equality holds because, since $\A \models \psi(a',c')$, we have $\spank(a') = \spank(c')$; the inequality follows directly from the definitions of $\tau$, $\phi$, $\theta$, etc).
\end{proof}

\begin{Thm} \label{Tnearmodelcomplete}
The theory $T$ is near model complete.
\end{Thm}
\begin{proof}
Immediate from Lemma~\ref{defbledelta} and Proposition~\ref{nearmodelcomplete}.
\end{proof}

\subsection{Forking independence and ranks}

We now calculate $U$-ranks and Morley ranks of 1-types in the theory $T$. Below the symbol $\INDEP{}{}$ denotes forking independence for the theory $T$.

\begin{Def}
Let $A,B,C$ be $\cl_0$-closed subsets of $\bar A$ with $B \subset A,C$. 

We say that \emph{$A$ and $C$ are in free amalgam over $B$} if $A$ and $C$ are $\acl_{\Alg}$-independent over $B$ and $G \cap (A+C) = (G \cap A) + (G \cap C)$.

We say that \emph{$A$ and $C$ are in strong free amalgam over $B$} if $A$ and $C$ are in free amalgam over $B$ and $A+C$ is strong in $\bar A$.

\end{Def}

\begin{Prop} \label{indep}
Let $a \subset \bar A$ and $B,C$ be strong $\cl_0$-closed subsets of $\bar \A$ with $B \subset C$. Then, $a\INDEP{B}{C}$ if and only if $A := \scl(aB)$ and $C$ are in strong free amalgam over $B$.
\end{Prop}
\begin{proof}

Let $a,B,C$ be as in the statement of the lemma. Let $A:= \scl(aB)$. Throughout the following, we use Lemma~\ref{eq-acl} without explicit mention.

Let us first remark that if $A$ and $C$ are in strong free amalgam over $B$, then this determines $\tp(a/C)$ uniquely among the extensions of $\tp(a/B)$ over $C$. To see this, assume $A$ and $C$ are in strong free amalgam over $B$. Let $a'\subset \bar A$ be such that $\tp(a'/B) = \tp(a/B)$ and $A':=\tp(a'B)$ and $C$ are in strong free amalgam over $B$. Since $\tp(a'/B) = \tp(a/B)$, the map $a \mapsto a'$ extends to an isomorphism from $A$ to $A'$ over $B$. Since $A$ and $C$ are in free amalgam over $B$ and so are $A'$ and $C$, then the above isomorphism can be extended to an isomorphism from $A+C$ to $A'+C$ over $C$. Since $A+C$ and $A'+C$ are strong, this shows that $\tp(a/C) = \tp(a'/C)$.

In order to prove the proposition, it is sufficient to show the following:

\textbf{Claim:} Suppose $A$ and $C$ are in strong free amalgam over $B$. Then $\tp(a/C)$ \emph{does not split} over $B$. 

Indeed, by the stability of $T$, it follows from the claim that if $a$ and $C$ are in free amalgam over $B$, then $\tp(a/C)$ is the \emph{unique} non-spliting, and thus non-forking, extension of $\tp(a/B)$ over $C$. \footnote{For the uniqueness see Corollary 12.6 in \cite{PzMT}. There, \emph{son} means extension, \emph{special} means non-splitting, \emph{heir} means non-forking extension}.

\textbf{Proof of claim:}
Suppose $A$ and $C$ are in strong free amalgam over $B$. Let $p= \tp(a/C)$. Let $\sigma$ be an automorphism of (the induced structure on) $C$ over $B$. We want to see that $\sigma(p) = p$.

Since $C$ is strong, $\sigma$ is an elementary map. Therefore we can find an automorphism $\bar\sigma$ of $\bar A$ extending $\sigma$. Since $\sigma(p) = \tp(\bar\sigma(a)/C)$, what we want to see is that $\tp(\bar\sigma(a)/C) = \tp(a/C)$.

Let $A':=\bar\sigma(A)$. It is easy to see that $A' = \scl(\bar\sigma(a)B)$. Note that the elementary maps $\bar\sigma|_{A}$ and $\id_C$ coincide on $B$. Since $A$ and $C$ are in free amalgam, the union of $\bar\sigma |_A$ and $\id_C$ extends to an isomorphism from $A+C$ to $A'+C$. By assumption, $A+C$ is strong, and so is $A'+C$, for it is the image of $A+C$ under the automorphism $\bar \sigma$ of $\bar A$. We have thus found an elementary map over $C$ sending $a$ to $\bar\sigma(a)$. Hence $\tp(\bar\sigma(a)/C) = \tp(a/C)$, as required.
\end{proof}

\begin{Rem}
The above proof shows that types over $\acl$-closed sets (in the (real) base sort) are stationary. 
\end{Rem}

Before calculating Morley ranks for the theory $T$, we calculate the U-rank of 1-types over finitely generated sets of parameters. Since $T$ is superstable, for every type there is a finite set over which it does not fork; it follows that the U-rank of any type equals that of a restriction to a finitely generated set of parameters.

\begin{Lemm} \label{RU}
Let $a \in \bar A$ and let $B \subset \bar A$ be a strong finitely generated $\cl_0$-closed set.
\begin{enumerate}
\item If $\dd(a/B) = 0$, then $\RU(a/B)< \omega$.
\item If $\dd(a/B) = 1$, then $\RU(a/B) = \omega + m$ for some $m \in \omega$.
\item If $\dd(a/B) = 2$, then $\RU(a/B) = \omega \cdot 2$.
\end{enumerate}
\end{Lemm}
\begin{proof}
Let $A := \scl(Ba)$. Note that $\delta(A/B) = \dd(a/B)$.


\textbf{Proof of 1.} Suppose $\dd(a/B) = 0$. 

If $a \in \acl(B)$, then $\RU(a/B) = 0$. Thus, assume $a \not\in \acl(B)$.

Let us assume first that the extension $B \str A$ is minimal. We shall see that in this case $\RU(a/B) = 1$.

Let $C$ be a finitely generated strong $\cl_0$-closed set containing $B$. We want to see that if $\tp(a/C)$ forks over $B$, then $\tp(a/C)$ is algebraic.

Suppose $\tp(a/C)$ is not algebraic, i.e. $a \not\in\acl(C)$. Then, by the minimality assumption, $A \cap \acl(C) = B$. 

It follows that $A$ and $C$ are $\acl_{\Alg}$-independent over $B$: To see this, let $D := \acl_{\Alg}(A) \cap \acl_{\Alg}(C)$. We want to show $D = \acl_{\Alg}(B)$. Note that $\acl(B) \subset \acl(D) \subset \acl(A) = \acl(Ba)$, therefore either $\acl(D) = \acl(B) = \acl_{\Alg}(B)$ or $\acl(D) = \acl(Ba) = \acl(A)$. But the latter is impossible, because then $A \subset \acl(D)$ and hence
\[
A = A \cap \acl(D) \subset A \cap \acl(\acl(C)) = A \cap \acl(C) = B,
\] 
a contradiction. Therefore the former holds and hence $D \subset \acl_{\Alg}(B)$. Since the inclusion from right to left is obvious, we get $D = \acl_{\Alg}(B)$. 

Moreover, $A$ and $C$ are in free amalgam over $B$: Indeed, using the modularity of $\cl_0$ and the $\acl_{\Alg}$-independence of $A$ and $C$ over $B$, one sees that:
\[
\delta(A/C) = \delta(A/B) - \ld((A+C)\cap G/A \cap G+ C\cap G).
\]
Since $\delta(A/B) = 0$ and $\delta(A/C) \geq 0$ (because $C$ is strong), we get that $\ld((A+C)\cap G / A\cap G+C\cap G) = 0$, which means that $A$ and $C$ are in free amalgam over $B$.

Also, we see that $\delta(A+C/C) = 0$, hence $A+C$ is strong.

Therefore, $A$ and $C$ are in strong free amalgam over $B$, which means that $\tp(a/C)$ does not fork over $B$.

We have thus seen that $\RU(a/B) = 1$.

Without the minimality assumption, the extension $B \str A$ decomposes into a tower of prealgebraic minimal extensions $B = B_0 \str B_1 \str \dots \str B_n = A$. Using the additivity of finite U-ranks, one sees that $RU(a/B) = n$ (for each $i$, take an element $b_i$ in $B_i \setminus B_{i-1}$; note that $B_i$ is contained in $\acl(B_{i-1} b_i)$; then $RU(a/B) = \sum_i RU(b_i/B_{i-1}) = n$.)


\textbf{Proof of 2.} Suppose $\dd(a/B) = 1$.

Assume the extension $B \str A$ is minimal. We want to see that then $\RU(a/B) = \omega$.

Let $C$ be a finitely generated strong $\cl_0$-closed set containing $B$. We claim that if $\tp(a/C)$ forks over $B$, then $\dd(a/C) = 0$.

Suppose $\dd(a/C) \neq 0$. Then $\dd(a/C) = 1$, for $\dd(a/C) \leq \dd(a/B) = 1$.
Clearly, $a \not\in\acl(C)$. As in the proof of part 1., it follows that $A$ and $C$ are $\acl_{\Alg}$-independent over $B$.
Also, since
\[
\delta(A/C) = \delta(A/B) - \ld((A+C)\cap G/A \cap G+ C\cap G),
\]
we have:
\begin{align*}
\ld((A+C)\cap G/A \cap G+ C\cap G) &= \delta(A/B) - \delta(A/C)\\ 
                                   &= \dd(a/B) - \delta(a/C)\\ 
                                   &\leq \dd(a/B) - \dd(a/C)\\
                                   &= 1 - 1\\
                                   &= 0.
\end{align*}
Therefore $A$ and $C$ are in free amalgam over $B$. From this we also get, 
\[
\delta(A+C/C) = \delta(A/B) = 1 = \dd(a/C), 
\]
which implies that $A+C$ is strong. Thus, $A$ and $C$ are in strong free amalgam over $B$, which means that $\tp(a/C)$ does not fork over $B$

Applying the result of the previous part, we obtain that all forking extensions of $\tp(a/B)$ have finite U-rank. Therefore $\RU(a/B) \leq \omega$.

By considering towers of minimal prealgebraic extensions, one sees that there are elements in $G$ of arbitrarily large U-rank over $B$ and, in fact, $\RU(a/B) = \omega$.

Without the minimality assumption, the extension $B \str A$ decomposes into a tower of minimal extensions, of which one is a minimal green generic extension and all the other are prealgebraic. From the previous arguments and Lascar's inequalities, we conclude that $\RU(a/B) = \omega + m$ for some natural number $m$.


\textbf{Proof of 3.} Suppose $\dd(a/B) = 2$.

Note that in this case $A = B + \spank(a)$ and the extension $B \str A$ is therefore minimal.

Let $C$ be a finitely generated strong $\cl_0$-closed set containing $B$. Let us see that if $\tp(a/C)$ forks over $B$, then $\dd(a/C) \leq 1$.

Suppose $\dd(a/C) > 1$, i.e. $\dd(a/C) = 2$. Clearly, $a \not\in\acl(C)$. As before, it follows that $A$ and $C$ are $T_{\Alg}$-independent over $B$.

As in the previous parts, we see that
\[
\ld((A+C)\cap G / A\cap G+C\cap G) \leq \dd(a/B) - \dd(a/C) = 2 - 2 = 0.
\]
Therefore $A$ and $C$ are in free amalgam over $B$. Then also $\delta(A+C/C) = \delta(A/B) = 2 = \dd(a/C)$, and hence $A$ and $C$ are in strong free amalgam over $B$.

Thus, all forking extensions of $\tp(a/B)$ have U-rank smaller than $\omega \cdot 2$. Therefore $\RU(a/B) \leq \omega \cdot 2$. 

Looking at towers of minimal extensions, this time adding a green generic extension at the top of the towers of prealgebraic extensions, one sees that, in fact, $\RU(a/B) = \omega \cdot 2$.
\end{proof}

\begin{Lemm} \label{RMRU}
For $T$, Morley rank and U-rank coincide on all 1-types.
\end{Lemm}
\begin{proof}
It is sufficient to prove this for global types, i.e. types over $\bar A$. 

It is well-known and easy to show that for any totally transcendental complete theory, for any (global) type $p$, $\RU(p) \leq \RM(p)$. In order to show that also $\RM(p) \leq \RU(p)$ for all $p$, it is sufficient to prove the following: for every $p \in S_1(\bar A)$, there is a formula $\psi \in p$ such that for all $q \in S_1(\bar A)$, if $\psi \in q$ and $\RU(q) \geq \RU(p)$ then $q=p$ (i.e. $\psi$ isolates $p$ among the types in $S_1(\bar A)$ with U-rank $\geq \RU(p)$).
\footnote{
Explanation for the sufficiency: Assume that for every $p \in S_1(\bar A)$, there is $\psi \in p$ such that $\psi$ isolates $p$ among the types in $S_1(\bar A)$ with U-rank $\geq \RU(p)$.
One can then show, by induction, that for every ordinal $\alpha$, if $\RM(p) \geq \alpha$ then $\RU(\alpha) \geq \alpha$: The cases where $\alpha = 0$ or $\alpha$ is a limit ordinal are easy and do not need the extra assumption. Suppose now that the implication holds for $\alpha$ and $\RM(p) \geq \alpha+1$. Since Morley rank coincides with the Cantor rank on $S_1(\bar A)$ (Proposition 17.17 in \cite{PzMT}), $p$ is an accumulation point of types $p_i \in S_1(\bar A)$ with $\RM(p_i) \geq \alpha$. By the induction hypothesis, $\RU(p_i) \geq \alpha$ and $\RU(p) \geq \alpha$. If $\RU(p) = \alpha$, the assumption yields that there is a formula that isolates $p$ from the $p_i$, hence a contradiction. Thus, $\RU(p) \geq \alpha +1$.
}

Let $p \in S_1(\bar A)$. We shall find $\psi \in p$ that isolates $p$ among the global types of greater or equal U-rank. Let $B \subset \bar A$ be the strong closure of a finite set over which $p$ does not fork. Let $a \in \bar A$ be a realisation of $p|_{\acl(B)}$. Let $A:= \scl(Ba)$. Note that it is sufficient to find a formula $\psi(x)$ over $\acl(B)$ that isolates $p|_{\acl(B)}$ among the types over $\acl(B)$ of greater or equal U-rank. This is what we do in each of the following cases.

\emph{Case 1: $\RU(a/B)$ is finite, i.e. $\dd(a/B) = 0$}. 

Assume $a \not\in\acl(B)$, otherwise the type of $a$ over $\acl(B)$ is obviously isolated. 

Let us assume first that the extension $B \str A$ is minimal. Since $\dd(a/b) = 0$, we can find a $\cl_0$-basis $a'$ of $A$ over $B$ with all coordinates in $G$. Then $a$ is algebraic over $a'$ in the language $L_{\Alg}$ and $\delta(a'/B) = 0$. Let $V(y,d)$ be the algebraic locus of $a'$ over $\acl(B)$. The variety $V(y,d)$ is a minimal prealgebraic rotund variety and therefore we can find a formula $\phi(y,d)$ for $V(y,d)$ as in Lemma~\ref{semidef-generic}. 
We claim that the formula $\phi(y,d) \land \bigwedge_i G(y_i)$ isolates the type of $a'$ over $\acl(B)$ among the types of greater or equal U-rank. Indeed, by the choice of $\phi(y,d)$ and the fact that $\acl(B)$ is strong, the following holds: for every $a''$  satisfying the formula $\phi(y,d)$, either $a'' \subset \acl(B)$ or $a''$ is a generic of $V(y,d)$ over $\acl(B)$. By the Thumbtack Lemma, we may assume that our $a'$ is Kummer generic over $\acl(B)$, which implies that every green generic of $V(y,d)$ over $\acl(B)$ has the same type as $a'$ over $\acl(B)$. Thus, we have: for every $a''$ satisfying $\phi(y,d)$, either $a'' \subset \acl(B)$ or $a''$ has the same type $a'$ over $\acl(B)$. This directly implies that $\phi(y,d) \land \bigwedge_i G(y_i)$ isolates the type of $a'$ over $\acl(B)$ among the types of greater or equal U-rank, as we claimed. 
Since $a$ is algebraic over $\acl(B) \cup a'$, we can find an formula $\theta(x,y)$ over $\acl(B)$ such that $\theta(x,a')$ isolates the type of $a$ over $a'\cup \acl(B)$. Then, using the additivity of finite U-ranks, the type of $a$ over $\acl(B)$ is isolated among the types of greater of equal U-rank by the formula $\psi(x) := \exists y (\theta(x,y) \land \phi(y,d))$.

Without the minimality assumption, the extension $B \str A$ decomposes into a tower of minimal prealgebraic extensions $B = A_0 \str \dots \str A_n = A$ with $a \in A_n \setminus A_{n-1}$.
For each $i=1,\dots, n$, let $a^i$ be a green basis of $A_i$ over $A_{i-1}$. Let $\theta(x,y^1,\dots,y^n)$ be a formula over $\acl(B)$ such that $\theta(x,a^1,\dots,a^n)$ isolates the type of $a$ over $\acl(B) \cup a^1 \cup \dots \cup a^n$, which is algebraic. Also, for each $i= 1,\dots,n$, let $\phi(y^i,d^{i-1})$ be a formula with $d^{i-1}\in \acl(A_{n-1})$ that isolates the type of $y^i$ over $\acl(A_{i-1})$ among the types of greater or equal U-rank, which we have seen above is possible to find. Finally, for each $i = 1,\dots,n-1$, let $\theta_i(z^i,y^1,\dots,y^{n-1})$ be a formula over $\acl(B)$ such that $\theta_i(z^i,a^1,\dots,a^{n-1})$ isolates the type of $d^i$ over $\acl(B) \cup a^1,\dots, a^{n-1}$, which is algebraic. Take $\psi(x)$ to be the following formula over $\acl(B)$:
\begin{multline}
\exists y^1 \dots \exists y^n 
\Big(
\theta(x,y^1,\dots,y^n) \land\\
\exists z^1 \dots \exists z^{n-1} 
(\bigwedge_{i=1}^n \phi(y^i,d^{i-1}) 
\land \bigwedge_{i=1}^{n-1} \theta_i(z^i,y^1,\dots,y^{n-1}))
\Big).
\end{multline}
Again, using the additivity of finite U-ranks, one sees that the formula $\psi(x)$ isolates the type of $a$ over $\acl(B)$ among the types of greater or equal U-rank.

\emph{Case 2: $\RU(a/B) = \omega +m$ for some $m \in \omega$, i.e. $\dd(a/B) = 1$}.

Note that if $a \in G$, then $\delta(a/B) = 1 = \dd(a/B)$, and hence $B + \spank(a)$ is strong and $\RU(a/B)=\omega$. Therefore, if $a \in G$, then $\psi(x)$ can be taken to be $G(x)$. 

More generally, whenever the extension $B \str A$ is minimal, $B + \spank(a)$ is strong and (although $a$ need not be in $G$) there is $a'$ in $G$ with $\RU(a'/B) = \omega$ such that $a'$ is a multiple of $a$. In particular, $a$ is algebraic over $a'$ and, a fortiori, over $\acl(B) \cup a'$.  Let $\theta(x,y)$ be a formula over $\acl(B)$ such that $\theta(x,a')$ isolates the type of $a$ over $\acl(B) \cup a'$. Then, the formula $\psi(x)$ can be taken to be $\exists y (\theta(x,y) \land G(y))$.

For arbitrary $a$ with $\dd(a/B) = 1$, the extension $B \str A$ decomposes into a tower of minimal extensions $B = A_0 \str \dots \str A_n = B$, where one is a green generic minimal extension and all others are prealgebraic. By the arguments up to this point, each of these minimal extensions can be dealt with and it is easy to see that the different formulas can be combined as in Case 1 to obtain an appropriate $\psi(x)$.

\emph{Case 3: $\RU(a/B) = \omega \cdot 2$, i.e. $\dd(a/B) = 2$} 

Since there is only one global type of U-rank $\omega \cdot 2$, the formula $\psi(x)$ can be taken to be $x=x$.
\end{proof}

\begin{Thm} \label{Ranks}
For $T$, the Morley rank of the universe is $\omega \cdot 2$ and the Morley rank of $G$ is $\omega$.
\end{Thm}
\begin{proof}
From the definitions we know that the Morley rank of a definable set is the maximum of the Morley ranks of the types containing a defining formula for the set. Also, Morley rank and U-rank coincide on all 1-types, by Lemma~\ref{RMRU}, and their values are as in Lemma~\ref{RU}. Thus, we have:
\begin{align*}
\RM(x=x) &= \max \{ \RM(p)\ :\ p \in S_1(\bar A),\ x =x \in p\} = \omega \cdot 2,\\
\RM(G(x))&= \max \{ \RM(p)\ :\ p \in S_1(\bar A),\ G(x) \in p\} = \omega.
\end{align*}
\end{proof}


\section{A variant: Emerald points} \label{sec:eme}

In this section we present a variation of the theories of green points in the multiplicative group case. The models of the new theories are expansions of the algebraic structure on the multiplicative group by a subgroup $H$ which is elementarily equivalent to the additive group of the integers \footnote{More generally, one can allow $H$ to have torsion and rather require $H/\Tor H$ to be elementarily equivalent to the additive group of the integers. For the sake of simplicity of notation we shall only work in the case where $H$ is torsion-free, but the more general case is not any more difficult.}. This contrasts with the case of green points where $G$ is divisible. In the new structures we call the elements of the subgroup $H$ \emph{emerald points}; as before, we call the elements outside the distinguished subgroup \emph{white points}. 

The motivation for our interest in these structures comes from Zilber's investigations on connections between model theory and noncommutative geometry; in particular, the content of \cite{ZNCG} and the survey \cite{ZToriSurvey}. In \cite{ZNCG},  a connection is established between the construction of noncommutative tori, which are basic examples of non-commutative spaces, and the model theory of the expansions of the complex field by a multiplicative subgroup of the form
\[
H = \exp(\eps \R + q \Z),
\]
where $\eps = 1 + i \beta$ and $\beta$ and $q$ are non-zero real numbers such that $\beta q$ and $\pi$ are $\Q$-linearly independent (this implies that $H$ is torsion-free).
In a subsequent joint paper with Boris Zilber we show that, assuming a certain consequence of the Schanuel Conjecture (the Schanuel Conjecture for raising to powers in the field $\Q(\beta i)$), these structures are indeed models of the theories of emerald points constructed here.

In what follows, the theories of emerald points are obtained by modifying the construction of the theories of green points presented in the previous sections. As in Section~\ref{sec:mth}, the theories are shown to be near model complete and superstable; they are not, however, $\omega$-stable.\\


\noindent\textbf{$\Z$-groups.}
Before going into the details, let us review some basic facts about the model theory of the theory of the additive group of the integers, whose interpretability in the theories of emerald points is responsible for the differences with respect to the green case.
The theory of the additive group of the integers eliminates quantifiers after adding to the semigroup language $\{+,0\}$ a predicate $P_m$ for each subgroup of the form $m\Z$, $m \geq 2$, and a constant for $1$. In this expanded language, a structure $(H,\cdot, 1,(P_m), \e)$
is elementarily equivalent to $(\Z,+,0,(m\Z),1)$ if and only if the following conditions are met: 
\begin{enumerate}[(i)]
\item $(H,\cdot, 1)$ is a torsion-free abelian group,
\item for every $m \geq 2$, $P_m$ is the set $H^m$ of all $m$-powers in $H$,
\item for every $m \geq 2$, $\e H^m$ generates the quotient group $H/H^m$. 
\end{enumerate}
If $(H,\cdot, 1,(P_m), \e)$ is elementarily equivalent to $(\Z,+,0,(m\Z),1)$, then it is said to be a \emph{$\Z$-group}. We also call this expanded language the \emph{language of $\Z$-groups}.
Note that every congruence equation in the integers of the form
\[
x \equiv k \pmod m,
\]
where $x$ is a variable, $m$ is a positive integer and $k \in \{0,\dots,m-1\}$, has a corresponding congruence equation 
in any $\Z$-group $(H,\cdot, 1,(P_m), \e)$, that we write multiplicatively as
\[
x \equiv \e^k \pmod m,
\]
and is expressed by the quantifier-free formula $P_m(x \e^{m-k})$. Also, every congruence equation in the integers of the form 
\[
t \equiv t' \pmod m,
\]
where $t$ and $t'$ are terms in the language of $\Z$-groups and $m$ is a positive integer, is equivalent to a Boolean combination of congruence equations of the above simpler form.
Moreover, the complete (quantifier-free) type of an element in a $\Z$-group is determined by the set of congruence equations, of the simple form above, that it satisfies. With this, it is easy to see that the theory of $\Z$-groups is $\lambda$-stable if and only if $\lambda \geq 2^{\aleph_0}$. The theory is thus superstable, non-$\omega$-stable. 

With the above remarks in mind and in order to obtain the same quantifier elimination results for the theories of emerald points as in the green case, we work in a language $L$ which in addition to $L_{\Alg} \cup \{ H\}$ also contains predicates $P_m$, for $m \geq 2$, and a constant $\e$; and require the expansion $(H,\cdot, 1,(P_m), \e)$ of the subgroup $(H,\cdot, 1)$ of $A$ to be a $\Z$-group. The obvious limitation with respect to the green case is that our theories of emerald points cannot be $\lambda$-stable for any $\lambda< 2^{\aleph_0}$; in particular, they cannot be $\omega$-stable. This is in fact the the only limitation in terms of stability, for the theories are in fact $\lambda$-stable for all $\lambda \geq 2^{\aleph_0}$, and hence superstable.

\subsection{Structures} 


In this section, $\Alg$ is the multiplicative group and we change to multiplicative notation.

Let $L$ be the language $L_{\Alg} \cup \{H, (P_m)_{m \geq 2}, \e\}$, where $H$ and each $P_m$ are unary predicates and $\e$ is a constant.

Let $\CC$ be the class of structures $\A = (A,H,(P_m)_{m \geq 2},\e)$ where
\begin{itemize}
\item $A$ is a model of $T_{\Alg}$,
\item $H$ is a subgroup of $A$,
\item $(H, \cdot, 1, (P_m)_{m \geq 2}, \e)$ is a $\Z$-group. 
\end{itemize}

If $\A$ is a structure in $\CC$, then we call the elements of $H$ \emph{emerald points}. The elements of $A \setminus H$ are called \emph{white points}.

For each $\A$ in $\CC$ we have:
\begin{itemize}
\item a pregeometry $\cl_0^\A$ on $A$ induced by the $\Q$-linear span pregeometry on the $\Q$-vector space $A/\Tor A$, and
\item a submodular predimension function $\delta_H^\A$ with respect to $\cl_0^\A$ defined by: for any finite dimensional $\cl_0^{\A}$-closed $Y \subset A$, 
\[
\delta_H^\A(Y) := 2 \trd(Y) - \ld(Y \cap \cl_0(H)).
\]
\end{itemize}
As in Section~\ref{sec:str}, the above definitions of $\cl_0^\A$ and $\delta_H^\A$ do not actually depend on the ambient structure $\A$ and we thus drop the superindices.


Let $\Sub\CC$ be the class of substructures of structures in $\CC$ whose domain is a $\cl_0$-closed set and let $\Fin\CC$ be the class of structures in $\Sub\CC$ whose domain has finite $\cl_0$-dimension.

\begin{Rem}
It is useful to note that if $H$ is a $\Z$-group and
$D$ is a divisible torsion-free abelian group, then the direct sum of
$H$ and $D$ is a $\Z$-group (with the obvious interpretations for the symbols of the expanded language). Also, if $H = (H,\cdot,1,(P_m)_M,\e)$ is a $\Z$-group and $D$ and $D_0$ are divisible torsion-free abelian groups such that $H$ and $D_0$ are subgroups of $D$ and $\e\in D_0$, then the intersection of $H$ and $D_0$ is a $\Z$-group.

It follows from the latter observation that for every $\X = (X,H^{\X})$ in $\Sub\CC$, $H^{\X}$ is a $\Z$-group.
\end{Rem}

Let us fix $\X_0 \in \Fin\CC$ and let $\CC_0$ be the class of structures in $\CC$ in which  $\X_0$ embeds strongly. After an identification, every structure in $\CC_0$ is assumed to have $\X_0$ as a strong substructure. 

Also, let $\Sub\CC_0$ be the class of substructures of structures in $\CC_0$ whose domain is a $\cl_0$-closed set containing $X_0$. Equivalently, $\Sub\CC_0$ is the class of structures in $\Sub\CC$ in which $\X_0$ embeds strongly, again identifying $\X_0$ with its image under one such embedding. Finally, let $\Fin\CC_0$ be the class of structures in $\Sub\CC_0$ whose domain has finite $\cl_0$-dimension.

As in Section~\ref{sec:str}, we work in the class $\Sub\CC_0$ and to simplify the notation we shall write simply $\cl_0$ for the localisation $(\cl_0)_{X_0}$ and $\delta_H$ for $(\delta_H)_{X_0}$. We also use the notation $\spank(X)$ for $\cl_0(X)$.


Note that the $L$-structure $\A_0$ with domain $A_0 = \acl_{\Alg}(X_0)$, $H^{\A_0} = H^{\X_0}$ and the obvious interpretations for the other symbols is in $\CC_0$ and is prime in $\CC_0$ with respect to strong embeddings.

\begin{Rem} \label{rem-dh-dg}
Suppose $\X = (X,H)$ is a structure in $\Sub\CC$. Consider the structure $\X' = (X,G)$, in the expansion of the language $L_{\Alg}$ by a unary predicate, where $G$ is a divisible subgroup of $X$ with $H \subset G \subset \cl_0(H)$. It is clear that $\X'$ is then a structure in $\Sub\CC$ in the sense of the green points construction. Also, if $\X$ is in $\CC$ or $\Fin\CC$, then $\X'$ is in the corresponding class in the sense of the green points construction.

Furthermore, note that for all finite-dimensional $\cl_0$-closed subset $Y$ of $X$, we have
\[
\delta_H^{\X}(Y) = \delta_G^{\X'}(Y).
\]
It follows that for all $\X,\Y \in \Sub\CC$ with $\X \subset \Y$, if $\X'$ and $\Y'$ are such that $G^{\X'} \subset G^{\Y'}$ and $\Tor G^{\X'} = \Tor G^{\Y'}$, then: $\X \str \Y$ if and only if $\X' \str \Y'$. This shows that the new predimension function $\delta_H$ is in a strong sense the same as the predimension function $\delta_G$, which is the reason why many results transfer effortlessly to the new setting.
\end{Rem}

\begin{Rem} \label{rich-em}
Using the previous remark, it is easy to see that the following statements follow from their analogues in the green points construction: the (Asymmetric) Amalgamation Lemma for the class $\Sub\CC_0$ (Lemma~\ref{aap-cc0} and Corollary~\ref{ap-cc0}), the closure of the class $\CC_0$ under unions of strong chains (Lemma~\ref{uc-cc0}), and the extension property from structures in $\Sub\CC_0$ to structures in $\CC_0$ (Lemma~\ref{ext-cc0}).

Thus, as in Section~\ref{sec:str},  we get the existence of rich structures in the class $\CC_0$. Also, just as in that section, we have that all rich structures are models of the same complete first-order theory and all partial isomorphisms between rich structures with strong domain and image are elementary maps.
\end{Rem}


\subsection{Theories}


The following two lemmas are directly implied by the corresponding ones in the green case.

\begin{Lemm} \label{lemm-ss-def-eme} 
Let $\A = (A,H) \in \CC$. For every complete $L_{\Alg}$-$l$-type $\Theta(y)$, there exists a partial $L_{\Alg}\cup\{H\}$-$l$-type $\Phi_\Theta(y)$ consisting of universal formulas such that for every realisation $c$ of $\Theta$ in $\A$, 
\[
\text{$\A \vDash \Phi_\Theta(c)$ if and only if $\spank c$ is strong in $\A$.}
\]
\end{Lemm}

\begin{Lemm} \label{lemm-ss-def2-eme} 
There exists an $L_{X_0}$-theory $T^0$ such that for every $L_{X_0}$-structure $\A = (A,H)$ in $\CC$,
$(A,H) \models T^0$ if and only if $(A,H)$ is in $\CC_0$.
\end{Lemm}

Henceforth let $T^0$ be a theory as in the above lemma. 


The following two definitions and the two subsequent lemmas are identical to their analogues in the green case.

\begin{Def}
An irreducible subvariety $W$ of $A^n$ is said to be \emph{rotund} if for every $k \times n$-matrix $M$ with entries in $\End(\Alg)$ of rank $k$, the dimension of the constructible set $M \cdot W$ is at least $\frac{k}{2}$.
\end{Def}

\begin{Def}
A structure $(A,H)$ in $\CC_0$ is said to have the \emph{EC-property} if for every even $n\geq 1$, for every rotund subvariety $W$ of $A^n$ of dimension $\frac{n}{2}$, the intersection $W \cap H^n$ is Zariski dense in $W$. 
\end{Def}

\begin{Lemm} 
For every subvariety $W(x,y)$ of $\Alg^{n+k}$ defined over $k_0$, there exists a quantifier-free $L_{\Alg}$-formula $\theta(y)$ such that for all $A\models T_{\Alg}$ and all $c \in A^k$,
\[
A \vDash \theta(c) \iff W(x,c)\text{ is rotund}.
\]
\end{Lemm}

\begin{Lemm}
There exists a set of $\forall\exists$-$L$-sentences $T^1$ such that for any structure $(A,H)$ in $\CC_0$
\[
(A,H) \vDash T^1 \iff (A,H)\text{ has the EC-property.}
\]
\end{Lemm}

Henceforth, let $T^1$ denote the theory defined in the above proof. Also, let $T := T^0 \cup T^1$. 

The next step is to show that the theory $T$ axiomatizes richness up to $\omega$-saturation. Here the difference with the green case lies in the following: in proving that $\omega$-saturated models of $T$ are rich, it is necessary to show that the following statement holds: 
\emph{
for any minimal prealgebraic strong extension $\X \str \Y$ of structures in $\Fin\CC_0$, the type of a coloured $\cl_0$-basis of $Y$ over $X$ is finitely satisfiable in models of $T$.
}
The difference with the green case is that, unlike there, here said type includes non-trivial information about the divisibility of the elements in the basis inside the coloured group $H$. Lemmas \ref{key-lemma0} and \ref{key-lemma1} provide the missing step in the argument by showing that the EC-property, and hence the axioms in $T$, are sufficiently strong to imply the above statement.

\begin{Lemm} \label{key-lemma0}
Let $\A$ be a structure in $\CC_0$.
The following are equivalent:
\begin{enumerate}
\item\label{ec1} $\A$ has the EC-property
\item\label{ec2} for every even $n\geq 1$, for every rotund subvariety $W$ of $A^n$ of dimension $\frac{n}{2}$ and every consistent system of congruence equations in the integers on variables $x_1,\dots,x_n$ of the form: 
\[
\{ x_i \equiv k_i \pmod{m_i} : i\in\{1,\dots,n\} \},
\]
with $m_i \geq 2$ and $k_i \in \{0, \dots, m_i\}$, the set of solutions in $H^n$ of the corresponding system in $H$ is Zariski dense in $W$.
\item\label{ec3} for every even $n\geq 1$, for every rotund subvariety $W$ of $A^n$ of dimension $\frac{n}{2}$ and every consistent system of congruence equations in the integers on variables $x_1,\dots,x_n$ of the form: 
\[
\{ x_i \equiv k_{im} \pmod m : i\in\{1,\dots,n\}, m\in\{2,\dots,N\} \},
\]
with $N \in \N$ and $k_{im} \in \{0,\dots,m\}$ , the set of solutions in $H^n$ of the corresponding system in $H$ is Zariski dense in $W$.
\end{enumerate}
\end{Lemm}
\begin{proof}
Clearly, (\ref{ec2}.) implies (\ref{ec1}.) and (\ref{ec3}.) implies (\ref{ec2}.).

Using the Chinese Remainder Theorem, it is easy to see that (\ref{ec2}.) implies (\ref{ec3}.).
 
We now show that (\ref{ec1}.) implies (\ref{ec2}.): 
Let $W$ be a rotund subvariety of $A^n$ of dimension $\frac{n}{2}$ and consider the system of congruence equations in the integers
\[
\{ x_i \equiv k_i \pmod{m_i} : i\in\{1,\dots,n\} \},
\]
where $m_i \geq 2$ and $k_i \in \{0,\dots,m_i-1\}$. Let $C$ be an $\acl_{\Alg}$-closed set over which $W$ is defined and let $a \in A^n$ be a generic point of $W$ over $C$. Let $a' \in A^n$ be such that $(a'_i)^{m_i} \e^{k_i} = a_i$, for $i=1,\dots, n$ and let $W'$ be the locus of $a'$ over $C$. It is easy to see 
that $W'$ is also a rotund subvariety of $A^n$ of dimension $\frac{n}{2}$. Assuming \ref{ec1} holds, the set $W' \cap H^n$ is Zariski dense in $W'$, which implies that the set of solutions of the corresponding system of congruence equations in $H$ is Zariski dense in $W$.
Thus, indeed, \ref{ec1} implies \ref{ec2}. 
\end{proof}
\begin{Lemm} \label{key-lemma1}
Let $\A$ be a structure in $\CC_0$. Assume $\A$ has the EC-property and is $\omega$-saturated. Then for every even $n\geq 1$, for every rotund subvariety $W$ of $A^n$ of dimension $\frac{n}{2}$ and every consistent system of congruence equations in the integers on variables $x_1,\dots,x_n$ of the form: 
\[
\{ x_i \equiv k_{im} \pmod m : i\in\{1,\dots,n\}, m\geq 2 \},
\]
where $k_{im} \in \{0,\dots,m\}$, the set of solutions in $H^n$ of the corresponding system in $H$ is Zariski dense in $W$.
\end{Lemm}
\begin{proof}
Follows immediately from the previous lemma.
\end{proof}

As noted before, with the above lemmas at hand we get the following proposition, just as in Section~\ref{sec:str}.

\begin{Prop} \label{iff-em}
The theory $T$ is complete and its $\omega$-saturated models are precisely the rich structures.
\end{Prop}


The next theorem gathers the main model-theoretic properties of the theory $T$.

\begin{Thm} \label{theorem:emerald1}
\begin{enumerate}
\item $T$ is near model complete
\item $T$ is $\lambda$-stable if and only if $\lambda \geq 2^{\aleph_0}$. The theory $T$ is therefore superstable, non-$\omega$-stable.
\item In $T$, $\RU(x=x) = \omega \cdot 2$ and $\RU(H(x)) = \omega$.
\end{enumerate}
\end{Thm}
\begin{proof}
For this proof, let $\bar \A = (\bar A,H)$ be a monster model for $T$.
\begin{enumerate}
\item Remark~\ref{rem-dh-dg} implies that with the same formulas $\tau_{a,b}$ as in the green case (see Lemma~\ref{defbledelta}), the sufficient condition for near model completeness found in Proposition~\ref{nearmodelcomplete} is satisfied.

\item Let $\lambda$ be an infinite cardinal and let $B \subset \bar A$ be set of cardinality $\lambda$. Let us show that there are at most $2^{\aleph_0}\cdot \lambda$ many 1-types over $B$. This clearly implies that $T$ is $\lambda$-stable for all $\lambda \geq 2^{\aleph_0}$. Note that $T$ is not $\lambda$-stable for any $\lambda < 2^{\aleph_0}$, for there are $2^{\aleph_0}$ 1-types over the empty set, as there are already in the theory of $\Z$-groups. 

By passing to the algebraic closure, we may assume that $B$ is algebraically closed, hence $\cl_0$-closed and strong. Let $a_0$ be an element of $\bar A$. Let $a\in \bar A^n$ be a $\cl_0$-basis of the set $A = \scl(B a_0)$ over $B$. For each coordinate $a_i$ of $a$ ($i=1,\dots,n$), define a sequence $(r_i^m)_{m \geq 1}$ as follows: $r_i^m := 0$ for all $m$, if $a_i \not\in H$, and, $r^m :=$ the remainder of $a_i$ modulo $m$ in $H$, if $a_i \in H$. Also, for each $m\geq 1$, let $a_i^m$ be 
a choice of $m^{th}$-root of $a_i (r_i^m)^{-1}$ in $H$, satisfying the compatibility condition that for all $m,m' \geq 1$, $(a_i^{mm'})^m = a_i^{m'}$. 

By Proposition~\ref{iff-em} and Remark~\ref{rich-em}, the type of $a_0$ over $B$ is determined by the $L$-isomorphism type of the set $A$ over $B$. The $L$-isomorphism type of $A$ over $B$ is itself determined by the following data, collectively: the set of indices $i \in \{1,\dots,n\}$ for which $a_i$ is in $H$, the sequence $(r_m)_{m \geq 1}$ and the algebraic type of the sequence $(a^m)_{m \geq 1}$ over $B$ (where, as usual, $r_m$ is the tuple with coordinates $r_i^m$ and $a^m$ the tuple with coordinates $a_i^m$).
There are finitely many possibilities for the set of indices, $2^{\aleph_0}$ possibilities for the sequence $(r_m)_{m \geq 1}$ and at most $\lambda \cdot 2^{\aleph_0}$ possibilities for the algebraic type of the sequence $(a^m)_{m \geq 1}$ over $B$. Thus, there are at most $\lambda \cdot 2^{\aleph_0}$ possibilities for the type of $a_0$ over $B$.

\item This is because the U-rank calculations work in the same way as in the green case. However, since the theory is not $\omega$-stable, Morley rank is not bounded; in particular, it does not coincide with $U$-rank. This is due to the interpretability of the theory of $\Z$-groups, where the same occurs. Note that the isolation of types among those of greater or equal U-rank, which we established in the green case to prove that Morley and U- ranks coincide, does indeed fail for the theory of $\Z$-groups.
\end{enumerate}
\end{proof}






\bibliography{jd-bib}
\end{document}